\documentclass[11pt]{amsart}

\usepackage{latexsym}
\usepackage{amssymb}
\usepackage{amsmath}
\usepackage{color}

\newtheorem{theorem}{Theorem}[section]
\newtheorem{lemma}[theorem]{Lemma}
\newtheorem{proposition}[theorem]{Proposition}
\newtheorem{corollary}[theorem]{Corollary}
\newtheorem{definition}[theorem]{Definition}

\newtheorem{examples}[theorem]{Examples}
\newtheorem{remark}[theorem]{Remark}

\newcommand\supp{\mathop{\rm supp}}

\newcommand\id{\mathop{\rm id}}
\newcommand\tr{\mathop{\rm tr}}

\newcommand\nph{\varphi}

\newcommand\nul{\mathop{\rm null}}
\newcommand\g{\mathfrak{g}}

\newcommand\vn{\mathop{\rm VN}}

\newcommand\psm{\mathop{\rm PM}}
\newcommand\psf{\mathop{\rm PF}}
\newcommand\nn{\mathop{\rm N}}

\newcommand{\cl}[1]{\mathcal{#1}}
\newcommand{\bb}[1]{\mathbb{#1}}

\begin{document}

\title{Reduced spectral synthesis and compact operator synthesis}

\author{V. S. Shulman}
\address{Department of Mathematics, Vologda State Technical University, Vologda, Russia}
\email{shulman.victor80@gmail.com}

\author{I. G. Todorov}
\address{Mathematical Sciences Research Centre, Queen's University Belfast, Belfast BT7 1NN, United Kingdom,
and
School of Mathematical Sciences, Nankai University, 300071 Tianjin, China}
\email {i.todorov@qub.ac.uk}

\author{L. Turowska}
\address{Department of Mathematical Sciences,
Chalmers University of Technology and  the University of Gothenburg,
Gothenburg SE-412 96, Sweden}
\email{turowska@chalmers.se}



\date{9 January 2019}

\begin{abstract}
We introduce and study the notion of reduced spectral synthesis, which
unifies the concepts of spectral synthesis and uniqueness in locally compact groups.
We exhibit a number of examples
and prove that every non-discrete locally compact group with an open abelian subgroup
has a subset that fails reduced spectral synthesis.
We introduce compact operator synthesis as an operator algebraic counterpart of this notion
and link it with other exceptional sets in operator algebra theory, studied previously.
We show that a closed subset $E$ of a second countable locally compact group $G$
satisfies reduced local spectral synthesis if and only if the subset $E^* = \{(s,t) : ts^{-1}\in E\}$ of $G\times G$
satisfies compact operator synthesis.
We apply our results to questions about the equivalence of linear operator equations with
normal commuting coefficients on Schatten $p$-classes.
\end{abstract}

\maketitle

\tableofcontents

\section{Introduction}

The concept of \emph{spectral synthesis}, arising from
approximation problems for functions defined on the real line,
is fundamental in classical Harmonic Analysis.
Research on the topic was fuelled in its initial stages by Schwartz' observation \cite{Schwartz} that the unit sphere in $\bb{R}^3$ does not satisfy spectral synthesis,
and the subsequent theorem of Malliavin's \cite{mall}, establishing the existence of non-synthetic sets in any non-discrete locally compact abelian group.
Later, the notion acquired prominence in Non-commutative Harmonic Analysis, where techniques from Functional Analysis played a
fundamental role.
Given a closed subset $E$ of a locally compact group $G$, there exist two extremal closed ideals $I(E)$ and $J(E)$ of the Fourier algebra $A(G)$ of $G$
that have null set $E$. The set $E$ is said to satisfy spectral synthesis if $I(E) = J(E)$.
Identifying the dual of $A(G)$ with the von Neumann algebra $\vn(G)$ of $G$ \cite{eym}, we thus have that $E$ is a set of spectral synthesis
precisely when the annihilators $I(E)^{\perp}$ and $J(E)^{\perp}$ in $\vn(G)$ coincide.

A connection between spectral synthesis and invariant subspace theory was pointed out by W. Arveson in \cite{arv}, and was later formalised, in the commutative case, by
J. Froelich \cite{froelich}
and, in the general locally compact  case,
by J. Ludwig and L. Turowska \cite{lutu},
leading to a rigorous link between spectral synthesis and an operator algebraic notion called
\emph{operator synthesis} \cite{st1}.
Namely, it was shown in \cite{lutu} that a subset $E$ of a locally compact second countable group $G$ satisfies
local spectral synthesis if and only if the subset $E^* = \{(s,t) : ts^{-1}\in E\}$ of $G\times G$ satisfies operator synthesis.

The concept of sets of \emph{uniqueness} in classical Harmonic Analysis, on the other hand, was motivated by questions about
uniqueness of Fourier series and was lifted to the non-commutative setting by M. Bo\.{z}ejko \cite{bozejko1}:
these are the subsets $E\subseteq G$ for which the intersection of the annihilator $J(E)^{\perp}$ with the reduced group C*-algebra $C^*_r(G)$ of $G$
is trivial.
Analogously to spectral synthesis, this concept has an operator theoretic counterpart, called \emph{operator uniqueness},
and a similar transference result holds true \cite{gralmul}.

In the present paper, we introduce a notion that unifies spectral synthesis and uniqueness,
which is new even in the classical, commutative, case.
Namely, we study the sets $E\subseteq G$ with the property that
$I(E)^{\perp} \cap C^*_r(G) = J(E)^{\perp} \cap C^*_r(G)$, which we call sets of
\emph{reduced spectral synthesis}.
We define an operator theoretic version of this concept, called henceforth
\emph{compact operator synthesis}, and establish a corresponding transference result.
We provide examples of both the failure and the validity of the properties of being a set of
reduced spectral synthesis or compact operator synthesis, and
apply our results to questions about equivalence of operator equations.

Our rationale behind investigating the link between reduced spectral synthesis and
compact operator synthesis is two-fold: on one hand, results from
Harmonic Analysis have been highly instrumental in providing examples of operator algebras or spaces
that have or fail a certain property of interest (see e.g. \cite{andersen, arv, dav-book, el, eks, hop-kraus, stt, st1});
on the other hand, results obtained using operator algebraic methods
have led, among others, to the identification of new classes of sets of spectral synthesis \cite{eletod,eletod2},
to unification and new proofs of transference results for
sets of uniqueness \cite{ivan-luda-un},
and to the introduction of new classes of multipliers of $A(G)$ \cite{gralmul}.

In more detail, the paper is organised as follows.
After collecting some preliminaries from Abstract Harmonic Analysis in Section \ref{s_p},
we introduce, in Section \ref{s_rss}, the notion of reduced spectral synthesis and,
extending a celebrated result of L. Schwartz \cite{Schwartz},
show that the unit sphere in $\bb{R}^n$ fails reduced spectral synthesis provided $n\geq 4$.
We also show that the light cone in $\bb{R}^{n+1}$ fails reduced spectral synthesis, and
investigate the functorial properties of this concept, which lead to the fact that
every non-discrete locally compact group that possesses an
open abelian subgroup has a subset that fails reduced spectral synthesis.

In Section \ref{s_cos}, we introduce compact operator synthesis,
the operator algebraic counterpart of reduced spectral synthesis.
This is a property of subsets $\kappa$ of the direct product $X\times Y$ of two (standard) measure spaces,
defined
by considering intersections of extremal masa-bimodules associated with $\kappa$ with the
space of all compact operators between the corresponding $L^2$-spaces of $X$ and $Y$.
We establish an Inverse Image Theorem for compact operator synthesis, and
link this concept to other classes of exceptional sets that have been studied previously
\cite{eletod, stt, gralmul}.
In Section \ref{s_c}, we show that a closed subset $E$ of a second countable locally compact group $G$
satisfies local reduced spectral synthesis if and only if the subset $E^*$ of $G\times G$ satisfies
compact operator synthesis.

In Section \ref{s_aoe}, we study a question arising from the classical Fuglede-Putnam Theorem.
Given commuting families $(A_i)_{i=1}^n$ and $(B_i)_{i=1}^n$ of normal operators on a Hilbert spaces $H$,
we consider the elementary operators
$\Delta$ and $\tilde{\Delta}$ on the space $\cl B(H)$ of all bounded operators on $H$, given by
$\Delta(T) = \sum_{i=1}^n A_i T B_i$ and $\tilde{\Delta}(T) = \sum_{i=1}^n A_i^* T B_i^*$.
We show that the equations $\Delta(T) = 0$ and  $\tilde{\Delta}(T) = 0$ are not equivalent,
when the operator are restricted to any of the Schatten classes $\cl C_p$, $p > 1$, or to the
the space of all compact operators on $H$.
We note that the question is inspired by the question of the equivalence of these equations on
the whole of $\cl B(H)$, which was answered negatively in \cite{sh_example}.
Our result herein improves the answer  given  in  \cite{st2} to a
similar question for operators on $\cl B(H)$ of infinite length.

\smallskip

If $\cl X$ is a Banach space and $\cl Y\subseteq \cl X$, we denote by
$\cl Y^{\perp}$ the annihilator of $\cl Y$ in the dual Banach space $\cl X^*$ of $\cl X$.
If there is a risk of confusion, the norm of $\cl X$ will be denoted by $\|\cdot\|_{\cl X}$.
Normed space duality will be denoted by $\langle\cdot,\cdot\rangle$, while
inner product in Hilbert spaces -- by $(\cdot,\cdot)$.


\section{Preliminaries}\label{s_p}

Let $G$ be a  locally compact group. Left Haar measure on $G$ will be denoted by $m$,
and integration with respect to $m$ along the variable $x$ will be denoted $dx$.
We write $L^p(G)$, $p = 1,2,\infty$, for the corresponding Lebesgue space
with respect to $m$, and $M(G)$ for the Banach algebra of all complex Borel measures on $G$.
We identify $L^1(G)$ with a (closed) ideal of $M(G)$ in the canonical fashion.
We let $C_0(G)$ (resp. $C_c(G)$) be the algebra of all continuous complex valued functions on $G$
vanishing at infinity (resp. having compact support).
As usual, the modular function of $G$ is denoted by $\Delta$.
Let $\lambda : G\rightarrow \cl B(L^2(G))$, $t\to \lambda_t$, be the left
regular representation; thus, $(\lambda_t\xi)(s)=\xi(t^{-1}s)$, $s,t\in G$, $\xi\in L^2(G)$.
We denote again by $\lambda$
the corresponding representation of $M(G)$ on $L^2(G)$;
for $\mu\in M(G)$, we have $\lambda(\mu)(g) = \mu \ast g$, $g\in L^2(G)$,
where $\mu\ast g$ is the convolution of $\mu$ and $g$, given by
$$(\mu\ast g)(t) = \int_G g(s^{-1}t) d\mu(s), \ \ \ t\in G.$$
The \emph{reduced group C*-algebra} of $G$ is the C*-subalgebra
$$C^*_r(G) = \overline{\{\lambda(f) : f\in L^1(G)\}}^{\|\cdot\|}$$
of $\cl B(L^2(G))$, while the
\emph{group von Neumann algebra} of $G$ is the von Neumann subalgebra
$\vn(G) = \overline{C^*_r(G)}^{w^*}$.

The \emph{Fourier-Stieltjes algebra} $B(G)$ of $G$
\cite{eym} is the algebra of
all functions $u : G\to \bb{C}$ of the form
\begin{equation}\label{eq_us}
u(s) = (\pi(s)\xi,\eta),
\end{equation}
where $\pi : G\rightarrow \cl B(H)$ is a continuous unitary representation and
$\xi, \eta\in H$. For $u\in B(G)$, its norm
$\|u\|_{B(G)}$ is, by definition, the infimum of the products $\|\xi\|\|\eta\|$
over all representations (\ref{eq_us}) of $u$.
The \emph{Fourier algebra} $A(G)$ of $G$ \cite{eym}
is the (commutative, regular, semi-simple) Banach algebra
consisting of
all functions $u : G\to \bb{C}$ of the form
\begin{equation}\label{ag}
u(x) = (\lambda_x \xi,\eta),\ \  x\in G,
\end{equation}
where $\xi,\eta\in L^2(G)$; we have that $A(G)$ is a closed ideal of $B(G)$.
The Banach space dual of $A(G)$ can be canonically identified with
$\vn(G)$  via the pairing $\langle u,T\rangle=(T\xi,\eta)$, where $u\in A(G)$ is given by
(\ref{ag}).
If $T\in \vn(G)$ and $u\in A(G)$, the operator $u\cdot T \in \vn(G)$ is
defined by the relations
$$\langle u\cdot T,v\rangle = \langle T,uv\rangle, \ \ \ v\in A(G).$$
The map $(u,T)\mapsto u\cdot T$ turns $\vn(G)$ into a left Banach $A(G)$-module.
Note that the inclusion $A(G)\subseteq C_0(G)$ gives rise to a canonical embedding $M(G)\subseteq \vn(G)$.
We refer the reader to \cite{eym} for further properties of $A(G)$ and $B(G)$.

If $J\subseteq A(G)$ is an ideal, let
$$\nul J = \{s\in G : u(s) = 0 \mbox{ for all } u\in J\}.$$
On the other hand, for a closed subset  $E\subseteq G$, let
$$I(E)=\{u\in A(G) : u(s)=0, s\in E\},$$
$$J_0(E) = \{u \in A(G) : u \text{ has compact support disjoint from } E\}$$
and
$J(E)=\overline{J_0(E)}.$
We have that
$$\nul J(E) = \nul I(E) = E$$
and that, if
$J\subseteq A(G)$ is a closed ideal with $\nul J = E$, then
$J(E)\subseteq J\subseteq I(E)$.

The {\it support} of an operator $T\in \vn(G)$ is the (closed) set
\begin{eqnarray*}
{\supp\mbox{}_{\vn}(T)} = \left\{ t\in G: u\cdot T\ne 0\text{  whenever } u\in A(G) \mbox{ and } u(t)\ne 0\right\}.
\end{eqnarray*}
It is known \cite{eym} that
$$J(E)^\perp = \{T\in \vn(G) : \supp\mbox{}_{\vn}(T)\subseteq E\}$$
and
$$I(E)^\perp = \overline{\{\lambda(\mu): \mu\in M(G), \supp(\mu)\subseteq E\}}^{w^*}
= \overline{\{\lambda_s : s\in E\}}^{w^*}.$$
A closed subset $E\subseteq G$ is called a {\it set of spectral synthesis} if $I(E) = J(E)$;
equivalently, $E$ is a set of spectral synthesis
if $\langle T,u\rangle=0$ for any $T\in \vn(G)$ with $\supp_{\vn}(T)\subseteq E$ and any $u\in I(E)$.

A closed subset $E\subseteq G$ is called a set {\it of uniqueness} (or an \emph{$U$-set})
if $C_r^*(G)\cap J(E)^\perp = \{0\}$, otherwise it is called a {\it set of multiplicity} (or an \emph{$M$-set}).
The set $E$ is called an \emph{$U_1$-set} if $C_r^*(G)\cap I(E)^\perp = \{0\}$; otherwise $E$ is called an \emph{$M_1$-set}.
We note that sets of uniqueness were extensively studied for the group of the circle $\mathbb T$ and the notion is closely related to
the convergence properties of Fourier series (see \cite{kl}).
For general locally compact groups, they were introduced by M. Bo\.{z}ejko \cite{bozejko}, and studied more recently in \cite{gralmul} (see also \cite{delaporte, derighetti}).

\section{Reduced spectral synthesis}\label{s_rss}

\subsection{Definition and examples}\label{s_rssde}

We start by defining one of the two main concepts that will be studied in this paper.

\begin{definition}\label{d_es}
A closed subset $E\subseteq G$ will be called a set of \emph{reduced spectral synthesis} if
$$C_r^*(G)\cap I(E)^\perp = C_r^*(G)\cap J(E)^\perp.$$
The set $E$ will be called a set of \emph{reduced local spectral synthesis} if
$\langle S,u\rangle=0$ for any  $S\in C_r^*(G)$ with $\supp_{\vn}(S)\subseteq E$ and any $u\in I(E)\cap C_c(G)$.
\end{definition}

Note that a closed subset $E\subseteq G$
is a set of reduced spectral synthesis if and only if
$\langle S,u\rangle = 0$ for all $S\in C_r^*(G)$ with $\supp_{\vn}(S)\subseteq E$ and all $u\in I(E)$.

Suppose that the group $G$ is abelian and let $\widehat G$ be its dual group.
We reformulate Definition \ref{d_es} for this case, using terminology widely accepted in the
literature on commutative harmonic analysis.
If $f\in L^1(\widehat{G})$, let
$\hat{f}$ be its Fourier transform; thus,
$$\hat{f}(s) = \int_{\widehat G} f(\gamma)\overline{\gamma(s)}d\gamma, \ \ \ s\in G.$$
We have that
$$A(G) = \left\{\hat{f} : f\in L^1(\widehat G)\right\}.$$
Let $\cl F : L^2(\widehat{G})\to L^2(G)$ be the (unitary) extension of the restriction of the
Fourier transform to $L^1(\widehat{G})\cap L^2(\widehat{G})$.
The space of \emph{pseudo-measures} $\psm(G)$ coincides, by definition, with the
dual Banach space of $A(G)$.
For $T\in \psm(G)$, we write $\supp_{\psm(G)}(T)$ for the support of the bounded functional $T$ acting
on the commutative Banach algebra $A(G)$ \cite{katznelson}.
Note that
$$\psm(G) = \left\{\cl F M_a \cl F^* : a\in L^\infty(\widehat G)\right\},$$
where $M_a$ stands for the operator on $L^2(\widehat{G})$ of multiplication by $a$.
The space of \emph{pseudo-functions} is defined by letting
$$\psf(G) = \left\{\cl F M_a \cl F^* : a\in C_0(\widehat G)\right\}.$$
Thus, $\psm(G)$ (resp. ${\rm PF}(G)$) can be naturally identified with $\vn(G)$ (resp. $C_r^*(G)$).

The inclusion $A(G)\subseteq C_0(G)$ gives rise to a canonical embedding $M(G)\subseteq \psm(G)$.
If $E$ is a closed subset of $G$, we will denote by
$\psm(E)$ (resp. $M(E)$) the space of all pseudo-measures $T\in \psm(G)$
(resp. measures $\mu\in M(G)$) with $\supp_{\psm(G)}(T)\subseteq E$
(resp. $\supp(\mu) \subseteq E$),
and let $\nn(E)$ be the weak* closure of $M(E)$.
Note that $\psm(E) = J(E)^\perp$ and $\nn(E) = I(E)^\perp$ (see e.g. \cite{gmcgehee}).
Thus, in the case where $G$ is abelian, the set $E$ is of reduced spectral synthesis
if and only if
$$\psm(E)\cap \psf(G) = \nn(E)\cap \psf(G),$$
that is, if and only if every pseudo-function supported by $E$ can be approximated in the weak* topology
by measures supported in $E$.

We proceed with a series of examples related to the notions introduced in Definition \ref{d_es}.

\begin{examples}\label{ex_ess}
\rm
{\bf (i) } Every set of spectral synthesis is trivially a set of reduced spectral synthesis.

\smallskip

{\bf (ii) } Every set of uniqueness  is a set of reduced spectral synthesis. Indeed, in this case,
$C_r^*(G)\cap I(E)^\perp = C_r^*(G)\cap J(E)^\perp = 0$.

\smallskip

{\bf (iii)}
Let $G = {\mathbb T}$, realised additively as ${\mathbb R}/2\pi\bb{Z}$.
For $2 < r < \infty$, let $E_r = \{\sum_{j=1}^{\infty}\varepsilon_jr^{-j}:\varepsilon_j=0,1\}$.
It is known (see \cite[p. 92]{gmcgehee}) that $E_r$ is an $U$-set if and only if $r$ is a Pisot number.
Moreover, by \cite[Theorem~3.3.2]{gmcgehee}, $E_r$ contains a subset $E(r)$
that fails spectral synthesis. Clearly, $E(r)$ is a set of uniqueness if $E_r$ is so.
Hence, if $r$ is a Pisot number then $E(r)$ is a set of reduced spectral synthesis but not of spectral synthesis.

\smallskip

{\bf (iv)}
Recall \cite{gmcgehee} that a closed subset $E$ of a locally compact abelian group $G$ is called
a \emph{Helson set} if every function in $C_0(E)$ is the restriction of a function from $A(G)$.
T. K\"orner \cite{korner} has constructed an example of a set of multiplicity $E\subseteq \mathbb T$
which is a Helson set, and therefore a $U_1$-set (see \cite[Section 4.5]{gmcgehee}). Clearly, $E$ fails reduced spectral synthesis.
\end{examples}

In the next two propositions, we will be concerned with the group $\bb{R}^n$. We identify it with its dual group and,
for elements $x,y\in \bb{R}^n$, write $x\cdot y = \langle x,y\rangle = \sum_{i=1}^n x_i y_i$,
where $z_i$ stands for the $i$-th coordinate of an element $z\in \bb{R}^n$, and
use $|x|$ to denote  the  Euclidean norm of $x$.
In his celebrated work \cite{Schwartz}, L. Schwartz  showed that the
Euclidean sphere $S^{n-1}$ in $\mathbb R^n$ fails spectral synthesis whenever $n\geq 3$.
On the other hand, N. Varopoulos proved \cite[Theorem 2]{varopoulos} that $S^2$ is a  set of
reduced spectral synthesis in ${\mathbb R}^3$.
Below  we show that this does not extend to the case where $n\geq 4$.
We will use distributions
and operations with them, such as derivatives and Fourier transforms;
we refer the reader to
\cite{Folland}, \cite{stein_weiss} and \cite{strichartz} for the necessary background.
We denote by  $C^{\infty}(\bb{R}^n)$
the space of all infinitely differentiable functions
on $\bb{R}^n$, and by $C_c^{\infty}(\bb{R}^n)$
the subspace of its compactly supported elements.
We equip $C_c^{\infty}(\bb{R}^n)$ with the topology given by the seminorms
$\|\cdot\|_{\alpha}$, where
$$\|f\|_{\alpha} = \max_{s\in \bb{R}^n} \left| \frac{\partial^{\alpha}f}{\partial^{\alpha}x}(s)\right|,$$
and $\alpha\in (\bb{N}\cup\{0\})^n$ is a multi-index.
The space of all distributions, that is,
all continuous linear functionals on $C_c^\infty(\mathbb R^n)$,
is usually denoted by $\mathcal D'(\mathbb R^n)$.

It is known that $C_c^\infty(\mathbb R^n)$ is a dense subset of
$A(\mathbb R^n)$, and the topology of $C_c^\infty(\mathbb R^n)$
is finer than that induced by $A(\mathbb R^n)$
(see e.g. \cite[3.26] {eym}).
Hence the elements of ${\rm PM}(\mathbb R^n)$ can be viewed as distributions on $\mathbb R^n$
in the natural way.

For $k > - \frac{1}{2}$, let $J_k$ be the Bessel function of order $k$, given by
$$J_k(x) = \frac{x^k}{2^k\Gamma\left(k+\frac{1}{2}\right)\Gamma\left(\frac{1}{2}\right)}
\int_{-1}^{1}e^{ixs}(1-s^2)^{(2k-1)/2} ds, \ \ \ x\in \bb{R}$$
(see \cite[p. 154]{stein_weiss}).
By \cite[Theorem 5.1]{Folland}, there exists $D_k > 0$ such that
\begin{equation}\label{eq_Bess}
|J_k(|y|)|\leq D_k|y|^{-1/2}, \ \ \ \ |y|>1.
\end{equation}

\smallskip

\begin{proposition}\label{noness}
The Euclidean sphere $S^{n-1}$ in ${\mathbb R}^n$ does not satisfy reduced spectral synthesis if $n\geq 4$.
\end{proposition}

\begin{proof}
Fix $n\geq 4$.
Let $\mu$ be the normalised surface area measure on $S^{n-1}$ and let $Q$ denote the distribution derivative
$\frac{\partial\mu}{\partial x_1}$, given by the formula
$$\langle Q,f\rangle = - \int_{\bb{R}^n} \frac{\partial f}{\partial x_1} d\mu, \ \ \ f\in C_c^{\infty}(\bb{R}^n).$$
By \cite[p. 154]{stein_weiss},
$$\hat{\mu}(t) = A|t|^{-(n-2)/2} J_{(n-2)/2}(|t|),$$
where
$A$ is a constant.
By (\ref{eq_Bess}), $\hat\mu(t)=O\left(\frac{1}{|t|^{(n-1)/2}}\right)$, as $|t|\to\infty$.
It follows that
\begin{equation}\label{eq_t1mu}
t_1\hat\mu(t)\in C_0({\mathbb R}^n).
\end{equation}
A direct verification now shows that
if $\hat g\in C_c^\infty(\mathbb R^n)$ then
\begin{equation}\label{eq_Qint}
\langle Q,  \hat g\rangle = \int_{\bb{R}^n} it_1\hat\mu(t) g(t) dt.
\end{equation}
It follows that $Q$ is bounded in the topology of $C_c^\infty(\bb{R}^n)$, induced by $A(\bb{R}^n)$,
and hence
it extends to a continuous functional on $A(\bb{R}^n)$ (denoted in the same way).
By continuity, (\ref{eq_Qint}) remains true for all $g\in L^1(\bb{R}^n)$
and now (\ref{eq_t1mu}) implies that $Q\in {\rm PF}(\mathbb R^n)$.

Since $\mu$ is supported in $S^{n-1}$, so is $Q$. In fact, if $v\in A(\mathbb R^n)$ vanishes in a neighborhood of $S^{n-1}$
and is compactly supported
then
there exists a sequence $(v_k)_{k\in \bb{N}}\subseteq A(\mathbb R^n)\cap C_c^\infty(\mathbb R^n)$ of functions,
vanishing in neighbourhoods of $S^{n-1}$, such that $\|v - v_k\|_{A(\mathbb R^n)} \to_{n\to\infty} 0$.
In fact,  choose non-negative functions $u_i\in C_c^\infty(\mathbb R^n)$, $i\in \bb{N}$, such that $\|u_i\|_1 = 1$,
$\supp(u_{i+1})\subseteq \supp(u_i)$ for each $i$, and $\cap_{i\in \bb{N}} \supp(u_i) = \{0\}$.
There exists $i_0$ such that that the function
$u_i\ast v\in C^\infty(\mathbb R^n)$ vanishes on a neighbourhood of $S^{n-1}$ whenever $i\geq i_0$.
Set $v_k = u_{i_0 + k}\ast v$, $k\in \bb{N}$.
If
$v=\xi\ast\eta$, $\xi, \eta\in L^2(\mathbb R^n)$, then $u_i\ast\xi\in L^2(\mathbb R^n)$,
$\|u_i\ast\xi-\xi\|_2\to_{i\to\infty} 0$ and hence $\|v-v_k\|_{A(\mathbb R^n)}\leq \|u_{i_0 + k}\ast\xi-\xi\|_2\|\eta\|_2\to_{k\to \infty} 0$.
Thus,
$$\langle Q,v\rangle
= \lim_{k\to\infty} \langle Q,v_k\rangle
= - \lim_{k\to\infty} \left\langle\mu,\frac{\partial v_k}{\partial x_1}\right\rangle = 0.$$
It now suffices to find a function $u\in A(\mathbb R^n)$ that vanishes on $S^{n-1}$ and
such that $\langle Q, u\rangle \ne 0$.
Let $c\in C_c^\infty(\mathbb R)$ with $c(1)\ne 0$ and set
$$u(x) = x_1c\left(|x|^2\right)\left(e^{-|x|^2+1} - e^{-2|x|^2+2}\right), \ \ \ x\in \bb{R}^n;$$
then $u(x) = 0$ whenever $x\in S^{n-1}$.
Since $u \in C_c^{\infty}(\mathbb R^n)$, we have that $u\in A(\mathbb R^n)$.
It remains to note that
$$\langle Q,u\rangle = -\left\langle\mu,\frac{\partial u}{\partial x_1}\right\rangle =
- c(1) \int_{\mathbb R^n} 2x_1^2d\mu\ne 0.$$
\end{proof}

In the next proposition we show that the light cone in $\mathbb R^{n+1}$ fails reduced  spectral synthesis.

\begin{proposition}\label{p_sec}
Let $n \geq 2$. Then
$B = \{(x,t)\in {\mathbb R}^n\times{\mathbb R} : |x| = t\}$
is an  $M_1$-set that fails reduced spectral synthesis for $\bb{R}^{n+1}$.
\end{proposition}

\begin{proof}
Write $G = {\mathbb R}^n\times{\mathbb R} = \{(x,t):x\in{\mathbb R}^n, t\in{\mathbb R}\}$,
equipped with Lebesgue measure $m$.
Let $\mu$ be the normalized surface area measure on the Euclidean sphere
$S^{n-1}$.
For $t > 0$ and a subset $C\subseteq \bb{R}^n$, write $t C = \{tx : x\in C\}$.
Let $\mu_t$ be the surface area measure on
the set $t S^{n-1} = \{x\in \bb{R}^n : |x| = t\}$; thus, if $C\subseteq S^{n-1}$ is a Borel set then
$\mu_t(tC) = t^{n-1}\mu(C).$
Fix a non-negative function $a\in C^\infty(\mathbb R)$ with $\supp(a) \subseteq [1/2,2]$ and
let $\nu$ be the Borel measure on $G$ given by
$$\nu(A) = \int_{0}^{+\infty} \mu_t(A_t)a(t)dt,$$
where $A_t = \{x\in \bb{R}^n : (x,t)\in A\}$, $t\geq 0$.
Then $\supp(\nu) \subseteq B$.
By  \cite[p. 154]{stein_weiss}, there exists $C_n>0$ such that
\begin{eqnarray*}
 \hat\nu(y,t)
 & = &
 \int_{0}^{+\infty} \left(\int_{\mathbb R^n}e^{-ix\cdot y}d\mu_s(x)\right)e^{-ist}a(s)ds\\
 & = &
 \int_{0}^{+\infty} \left(\int_{\mathbb R^n}e^{-isx\cdot y}d\mu(x)\right)e^{-ist}s^{n-1}a(s)ds\\
 & = &
 \int_{0}^{+\infty} \hat\mu(sy)s^{n-1}e^{-ist} a(s)ds\\
 & = &
 C_n \int_{0}^{+\infty} (s|y|)^{-(n-2)/2}J_{(n-2)/2}(s|y|)s^{n-1}a(s)e^{-ist}ds.
\end{eqnarray*}
By (\ref{eq_Bess}), there exists $\tilde D_n>0$ such that
\begin{equation}\label{eq_y}
\displaystyle|\hat\nu(y,t)|\leq \frac{\tilde D_n}{|y|^{(n-1)/2}}, \ \ \ |y|>2, \ t\in\mathbb R.
\end{equation}

Set
$$F_1(s) = \int_{-1}^{1} e^{is|y|r}(1-r^2)^{(n-3)/2} dr \ \mbox{ and } \ F_2(s) = s^{n-1}a(s), \ \ s\in \bb{R}^+.$$
For suitable positive constants  $C$, $\tilde C_n$ and $D$, using integration by parts, for $t\neq 0$ we have
\begin{eqnarray}\label{eq_t}
|\hat\nu(y,t)|
& = &
\tilde C_n \left| \int_{0}^{+\infty}
F_1(s) F_2(s) e^{-its}ds\right|
=
\frac{\tilde C_n}{|t|} \left| \int_{0}^{+\infty} (F_1F_2)'(s) e^{-its}ds\right|\nonumber\\
& \leq &
\frac{\tilde C_n}{|t|} \int_{0}^{+\infty} \left|\int_{-1}^{1}|y|r e^{is|y|r}(1-r^2)^{(n-3)/2} dr\right| |F_2(s)| ds \nonumber\\
& + &
\frac{\tilde C_n}{|t|} \int_{0}^{+\infty} \left|\int_{-1}^{1} e^{is|y|r}(1-r^2)^{(n-3)/2} dr\right| |F_2'(s)| ds \nonumber\\
& \leq &
\frac{D|y|+C}{|t|}.
\end{eqnarray}

Set $g(R) = \left((-1+\sqrt{1+4R^2})/2\right)^{1/2}$, $R > 0$,
and assume that $y\in \bb{R}^n$ and $t\in \bb{R}$ are such that $|y|^2+t^2\geq R^2$.
If $|y|\geq |t|^{1/2}$ then
$|y|^2+|y|^4\geq R^2$ and hence $|y|\geq g(R)$; by (\ref{eq_y}),
\begin{equation}\label{eq_yR}
|\hat\nu(y,t)|\leq \frac{\tilde D_n}{g(R)^{(n-1)/2}}.
\end{equation}
On the other hand, if
$|y|\leq |t|^{1/2}$ then $|t|+t^2\geq R^2$ and hence $|t|\geq g(R)^2$;
by (\ref{eq_t}),
\begin{equation}\label{eq_tR}
|\hat\nu(y,t)|\leq \frac{D}{|t|^{1/2}}+\frac{C}{|t|}\leq  \frac{D}{g(R)}+\frac{C}{g(R)^2}.
\end{equation}
Inequalities (\ref{eq_yR}) and (\ref{eq_tR}) now imply that $\hat\nu\in C_0(G)$.
In particular, $B$ is an $M_1$-set.

Let $Q = \partial\nu/\partial x_1$, considered as a pseudo-measure
(see the proof of Proposition \ref{noness}).
Note that, since $\nu$ is supported by $B$, so is $Q$.
Similarly to identity (\ref{eq_Qint}), we have that
$\hat Q(y,t) = y_1\hat\nu(y,t)$, and (\ref{eq_yR}) and (\ref{eq_tR}) show that $Q$ is in fact a
pseudo-function.
Let $c : \bb{R}\to \bb{R}^+$ be a non-zero infinitely differentiable function with compact support  and
$$u(x,t) = x_1c\left(|x|\right)c(t)
\left(e^{-|x|^2+t^2} - e^{-2|x|^2+2t^2}\right), \ \ \ (x,t)\in \bb{R}^n \times \bb{R}.$$
As in the proof of Proposition \ref{noness},
$u\in A(G)$, $u$ vanishes on $B$ and
$$\langle Q,u\rangle = -\left\langle \nu,\frac{\partial u}{\partial x_1}\right\rangle = -\int_{G} 2x_1^2c(t)^2d\nu(x,t).$$
By choosing $c$ such that the last integral is non-zero, we see that
$B$ is not a set of reduced spectral synthesis.
\end{proof}

We note that the set $B$ from Proposition \ref{p_sec}
is the boundary of the set $E=\{(x,t)\in\mathbb R^n\times\mathbb R:|x|\leq t\}$,
which is known to be  a set of spectral, and hence of reduced spectral, synthesis
for $\bb{R}^{n+1}$ \cite[Corollary, p. 498]{arv}, \cite{froelich}.

\subsection{Functorial properties}\label{s_fprss}

Let $G$ be a locally compact group and $H\subseteq G$ be a closed normal subgroup.
For the remainder of this section, we let $q : G\to G/H$ be the canonical quotient homomorphism.
We set $\dot{x} = q(x)$, $x\in G$.
Given a function $f \in C_c(G)$, let
$\phi(f) : G\to \bb{C}$ be the function given by
$$\phi(f) (x) = \int_H f(xy)dy, \ \ \ x\in G,$$
where the integral is with respect to normalised left Haar measure on $H$.
The function $\phi(f)$ is constant on the cosets of $H$, and we denote again by
$\phi(f)$ the induced function on $G/H$.
We will make use of Weil's quotient intergal formula \cite[Theorem 2.49]{Folland_HA}:
\begin{equation}\label{eq_weil}
\int_G f(s) ds = \int_{G/H} \phi(f)(\dot{x}) d\dot{x}, \ \ \ f\in C_c(G).
\end{equation}

The proof of the following lemma, which will be needed in the sequel, is straightforward and is omitted.

\begin{lemma}\label{l_2comp}
Let $G$ be a locally compact group,
$K\subseteq G$ be a closed  set and $L\subseteq G$ a compact set such that
$K\cap L = \emptyset$. Then there exists an open neighbourhood $V$ of $e$
such that $\overline{V}$ is compact and $KV\cap L = \emptyset$.
\end{lemma}

We will need the following fact, established in \cite[Lemma 1]{lohoue2}.

\begin{lemma}\label{l_lohoue}
Let $G$ be a locally compact group, $H\subseteq G$ be a closed normal subgroup
and $K\subseteq G$ be a compact set.
There exists $C_K > 0$
such that if $u\in A(G)$ is supported on $K$ then
$\phi(u)\in A(G/H)$ and $\|\phi(u)\|_{A(G/H)}\leq C_K \|u\|_{A(G)}$.
\end{lemma}

We will say that a closed normal subgroup $H$ of $G$ has \emph{property (l)} if
for every proper compact subset $K\subseteq G/H$ there exists
$\theta\in A(G)\cap C_c(G)$ such that $\phi(\theta) = 1$ on a neighbourhood of $K$.

\begin{theorem}\label{th_quogen}
Let $G$ be a locally compact group, $H\subseteq G$ be a closed normal subgroup with property (l) and
$E\subseteq G/H$ be a compact set.
If $q^{-1}(E)$ satisfies reduced spectral synthesis then $E$ does so as well.
\end{theorem}

\begin{proof}
Let  $\theta\in A(G)\cap C_c(G)$.
By Lemma \ref{l_lohoue}, given $T\in \vn(G/H)$ and $u\in A(G)$,  there exists $C > 0$ such that
\begin{equation}\label{eq_theta}
|\langle T,\phi(\theta u)\rangle| \leq C\|T\|\|\theta u\|_{A(G)}.
\end{equation}
Thus, the functional
$u \to \langle T,\phi(\theta u)\rangle$ on $A(G)$ is bounded, and hence there exists an operator
$\Phi_\theta(T)\in \vn(G)$ such that
\begin{equation}\label{eq_phik}
\langle \Phi_\theta(T), u\rangle = \langle T,\phi(\theta u)\rangle, \ \ u\in A(G).
\end{equation}
By (\ref{eq_theta}),
$$\|\Phi_\theta(T)\| \leq C\|T\|\|\theta\|_{A(G)}.$$

We claim that $\Phi_\theta$ maps $C_r^*(G/H)$ into $C_r^*(G)$.
Indeed, suppose that $f\in C_c(G/H)$
and let $\tilde{f} = f\circ q$. If $u\in A(G)$ then, using (\ref{eq_weil}), we have
\begin{eqnarray*}
\langle \lambda(f),\phi(\theta u)\rangle & = &
\int_{G/H} f(\dot{x})\left(\int_H \theta(xh)u(xh)dh\right) d\dot{x}\\
& = & \int_{G/H}\int_H \tilde{f}(xh) \theta(xh)u(xh)dh d\dot{x}\\
& = & \int_G \tilde{f}(s) \theta(s)u(s)ds = \langle\lambda(\theta\tilde{f}),u\rangle.
\end{eqnarray*}
Thus shows that $\Phi_\theta(\lambda(f)) = \lambda(\theta\tilde{f})$; since $\theta\tilde{f}\in C_c(G)$,
we conclude that $\Phi_\theta(\lambda(f))\in C_r^*(G)$.
Since $\Phi_\theta$ is bounded, $\Phi_\theta(T)\in C_r^*(G)$ for every $T\in C_r^*(G/H)$.

Set $\tilde{E}\stackrel{def}{=} q^{-1}(E)$ and
let $T\in C_r^*(G/H) \cap J(E)^{\perp}$.
We claim that $\Phi_\theta(T) \in C_r^*(G)\cap J(\tilde{E})^{\perp}$. Indeed, let
$u\in A(G)$ be a function with compact support such that the set $U := (\supp(u))^c$ contains $\tilde{E}$.
By Lemma \ref{l_2comp},
there exists an open neighborhood $V$ of $e$ in $G$ such that $V\tilde{E}\subseteq U$.
Note, moreover, that $V\tilde{E} = V\tilde{E}H$.
As the quotient map is open, $q(V\tilde{E})$ and $q(U)$ are open neighborhoods of $E$ and
$q(V\tilde{E})\subseteq q(U)$.
If $\dot{x}\in q(V\tilde{E})$ then $xh\in V\tilde{E}$ for all $h\in H$ and hence
$$\phi(\theta u)(\dot{x}) = \int_H \theta(xh) u(xh)dh = 0.$$
Thus, $\phi(\theta u)$ is supported by the compact set $q(\supp (u))$ and
vanishes on the neighbourhood $q(V\tilde{E})$ of $E$;
it follows from (\ref{eq_phik}) that $\langle\Phi_\theta(T),u\rangle = 0$ and hence $\Phi_\theta(T)\in J(\tilde{E})^{\perp}$.

Now suppose that $T\in C_r^*(G/H) \cap J(E)^{\perp}$, and $u\in I(E)$.
By the previous paragraphs,
\begin{equation}\label{eq_eqji}
\Phi_\theta(T)\in J(\tilde{E})^{\perp}\cap C_r^*(G) = I(\tilde{E})^{\perp}\cap C_r^*(G).
\end{equation}
Since $H$ has property (l), there exists a function $\theta \in A(G)\cap C_c(G)$
such that $\phi(\theta) = 1$ on a neighbourhood of $E$. Then
\begin{equation}\label{eq_Tphi}
T = \phi(\theta) \cdot T.
\end{equation}
By \cite[Corollary 2.26]{eym}, $u\circ q\in B(G)$.
Let $\theta'\in A(G)\cap C_c(G)$ be a function taking value $1$ on a neighbourhood of
$\supp(\theta)$.
Then $\theta(u\circ q), \theta'(u\circ q)\in B(G)\cap C_c(G)$ and so $\theta(u\circ q),\theta'(u\circ q)\in A(G)$.
In addition,
$$\phi(\theta' \theta(u\circ q))(\dot{x}) = u(\dot{x})\int_H \theta'(xh) \theta(xh)dh = u(\dot{x})\phi(\theta)(\dot{x}), \ \ \ x\in G.$$
As $u\circ q$ vanishes on $\tilde E$, (\ref{eq_eqji}) and (\ref{eq_Tphi}) now imply
$$\langle T,u\rangle = \langle T,u\phi(\theta)\rangle
 = \langle T,\phi(\theta\theta'(u\circ q))\rangle = \langle \Phi_\theta(T),\theta'(u\circ q)\rangle = 0.$$
Thus, $T\in I(E)^{\perp}$ and the proof is complete.
\end{proof}

\begin{corollary}\label{c_proess}
Let $G$ be a locally compact group, $H\subseteq G$ be a compact normal subgroup
and $E\subseteq G/H$ be a compact set.
If $q^{-1}(E)$ is a set of reduced spectral synthesis then $E$ is so as well.
\end{corollary}

\begin{proof}
By Theorem \ref{th_quogen}, it suffices to show that $H$ has property (l).
Given a compact set $K\subseteq G/H$, let $u \in A(G/H)$ be a compactly supported function
such that $u = 1$ on a neighbourhood of $K$.
Since $H$ is compact, $q^{-1}(\supp(u))$ is a compact set;
it clearly supports $u \circ q$.
By \cite[Proposition 3.25]{eym}, $u \circ q$ belongs to $A(G)$;
moreover, $\phi(u \circ q)(\dot{x}) = u(\dot{x})$ for every $\dot{x}\in G/H$,
showing that $\phi(u \circ q)=1$ on a neighbourhood of $K$.
\end{proof}

\begin{corollary}\label{l_prop}
Let $G_1$ and $G_2$ be locally compact groups and $E\subseteq G_1$ be a closed subset.
If $E\times G_2$ is set of reduced spectral synthesis for $G_1\times G_2$ then $E$ is a set of reduced spectral synthesis for $G_1$.
\end{corollary}
\begin{proof}
By Theorem \ref{th_quogen}, it suffices to show
that the subgroup $\{e\}\times G_2$ of $G_1\times G_2$ satisfies property (l).
Fix a compact subset $K$ of $G_1$, viewed as $\left(G_1\times G_2\right)/\left(\{e\}\times G_2\right)$.
Let $\theta_1\in A(G_1)\cap C_c(G_1)$ be such that $\theta_1(x) = 1$ on a neighbourhood of $K$,
and $\theta_2\in A(G_2)\cap C_c(G_2)$ be such that $\int_{G_2}\theta_2(y)dy = 1$.
Define
$\theta:G_1\times G_2\to\mathbb C$ by letting $\theta(x,y) = \theta_1(x)\theta_2(y)$, $x\in G_1$, $y\in G_2$.
Then $\theta\in A(G_1\times G_2)\cap C_c(G_1\times G_2)$ and
$$\phi(\theta)(x) = \theta_1(x)\int_{G_2}\theta(y)dy = \theta_1(x), \ \ \ x\in G_1.$$
\end{proof}

\begin{proposition}\label{l_closedsub}
Let $G$ be a locally compact group, $H\subseteq G$ be an open subgroup and
$E\subseteq H$ be a closed set. If $E$ is a set of reduced spectral synthesis for $G$,
then $E$ is a set of reduced spectral synthesis for $H$.
\end{proposition}

\begin{proof}
We denote by $I_H(E)$ and $J_H(E)$ the extremal ideals of $A(H)$, associated with the set $E$.
The restriction map $r$, given by $r(u) = u|_{H}$, is a
contractive surjection from $A(G)$ onto $A(H)$
(see  \cite[Section 3]{kaniuth-lau}).
Let $r^* : \vn(H)\to \vn(G)$ be its adjoint.
Since $H$ is open, $r^*(C_r^*(H))\subseteq C_r^*(G)$ (see \cite[Remark 3.3]{kaniuth-lau}).

Suppose that $T\in C_r^*(H)\cap J_H(E)^{\perp}$.
We claim that $r^*(T)\in J(E)^{\perp}$. To see this, let
$v\in A(G)$ be a function with compact support, vanishing on a neighbourhood $V$ of $E$ in $G$.
Then $r(v)$ vanishes on the neighbourhood $V\cap H$ of $E$ in $H$ and has compact support.
Thus,
$$\langle r^*(T),v\rangle = \langle T,r(v)\rangle = 0.$$
By the assumption, $r^*(T)\in I(E)^{\perp}$.
Let $u\in A(H)$ be a function vanishing on $E$. If $v\in A(G)$ is such
that $r(v) = u$, then
$$\langle T,u\rangle = \langle T,r(v)\rangle = \langle r^*(T),v\rangle = 0.$$
Thus, $T\in C_r^*(H)\cap I_H(E)^{\perp}$ and the proof is complete.
\end{proof}

\begin{remark}\label{rem}{\rm
{\bf (i) } Proposition \ref{l_closedsub} does not hold for
arbitrary closed subgroups $H$ of $G$.
In fact, by \cite[Corollary 5.10]{gralmul}, any such subgroup, and hence any of its closed subsets, is a set of uniqueness and therefore a set of
reduced spectral synthesis with respect to $G$.
In particular, any closed subset of $\mathbb R^n\times\{0\}$, $n\geq 4$, is a set of reduced spectral
synthesis for $\mathbb R^{n+1}$. On the other hand, by Proposition \ref{noness},
$S^{n-1}\subseteq \mathbb R^n$ is not a set of
reduced spectral synthesis for $\mathbb R^n$.

\smallskip

{\bf (ii) } It is well-known that the statement of Proposition \ref{l_closedsub} holds for any closed subgroup $H$,
when reduced spectral synthesis is replaced by spectral synthesis.
To see this, one applies similar arguments to the ones from the proof of Proposition \ref{l_closedsub}. }
\end{remark}

\subsection{Failure of reduced spectral synthesis}

A well-known theorem of P. Malliavin's \cite{mall} states
that $A(G)$ fails spectral synthesis whenever $G$ is a non-discrete abelian locally compact group.
The theorem was later reproved using the Varopoulos  tensor algebra techniques (see e.g. \cite[Theorem 11.2.1]{gmcgehee}).
Using results of B. Forrest \cite{forrest} and of E. Zelmanov \cite{zelmanov},
E. Kaniuth and A. T. Lau extended
Malliavin's theorem to arbitrary non-discrete locally compact groups \cite[Proposition 2.2]{kaniuth-lau}.
Here we prove a related result for reduced spectral synthesis.
We will need the following lemma.

\begin{lemma}\label{l_R}
The subgroup $\bb{Z}$ of $\bb{R}$ possesses property (l).
\end{lemma}

\begin{proof}
Let $q : \bb{R}\to \bb{R}/\bb{Z}$ be the quotient map, and identify $\bb{R}/\bb{Z}$ with $[0,1)$ in the
canonical way.
Fix a proper compact subset $K\subseteq [0,1)$; we assume that
$0\not\in K$, for otherwise we may replace $K$ by a suitable translation.
Let $\theta\in A(\mathbb R)$ be such that $\theta (x) = 1$ on a neighborhood $U$ of $K$ and $\theta(x)=0$ whenever $x\not\in (0,1)$.
If $x\in U$ then
$$\phi(\theta)(\dot{x})=\sum_{h\in\mathbb Z}\theta(x+h)=\theta(x)=1.$$
\end{proof}

\begin{theorem}\label{th_every}
Let $G$ be a non-discrete locally compact group that has an open abelian subgroup.
Then $G$ has a subset which fails reduced spectral synthesis.
\end{theorem}

\begin{proof}
By \cite[7.8.3-7.8.6]{rudin},
if $H$ is a compact abelian group and $p>2$,
there exists  a pseudo-measure $F\in {\rm PM}(H)$ such that
$\hat F\in \ell^p(\widehat H)$
and $F\not\in N(\supp(F))$.
It follows that the subset $\supp(F)$ of $H$ does not satisfy
reduced spectral synthesis. Taking $H = \bb{T}$, and applying
Theorem \ref{th_quogen} and Lemma \ref{l_R},
we conclude that there exists a closed subset of $\bb{R}$ that fails reduced spectral synthesis.
Since every non-discrete locally compact abelian group $G$ contains an open subgroup of the form
$\bb{R}^n\times H$, where $H$ is compact and $n\geq 0$ (see e.g. \cite[Theorem 2.4.1]{rudin}), we conclude  by
Corollary \ref{l_prop} and Proposition \ref{l_closedsub}
that every such $G$ has a subset that fails reduced spectral synthesis.
The claim follows from a further application of Proposition \ref{l_closedsub}.
\end{proof}

\begin{remark}\rm
S. Saeki \cite{saeki} has shown that any locally compact non-discrete metrisable abelian group
$G$ has a closed $U_1$-set that is not a set of uniqueness.
By \cite[Section 4.5]{gmcgehee}, this implies the statement of Theorem \ref{th_every} for this particular class of groups.
\end{remark}

\begin{theorem}\label{p_Lie}
Let $G$ be a simply connected, connected  nilpotent Lie group. Then $G$ has a subset which fails reduced spectral synthesis.
\end{theorem}

\begin{proof}
Let $\g$ be the Lie algebra of $G$ and $\exp:\g\to G$ be the exponential map.
Denote by  $[G,G]$ the commutator of $G$.
The assumptions imply that $G$ is exponential,
that is, $\exp$ is a diffeomorphism, and  we have that $[G,G]=\exp([\g,\g])$.

If $[G,G] = \{e\}$ then $G$ is abelian and connected. Hence, by Theorem \ref{th_every}, it fails reduced spectral synthesis.
Let $[G,G]\ne \{e\}$. Then $G/[G,G]\ne\{e\}$ is abelian and hence, by Theorem \ref{th_every},
has a closed subset that does not satisfy reduced spectral synthesis.
Therefore, in order to prove the statement, it suffices by Theorem \ref{th_quogen}
to show that $[G,G]$ has property (l).
Since $[\g,\g]$ is an ideal of $\g$, \cite[Corollary 6.1.4]{jean_book} implies that there exists
a Jordan-H\"older basis $\mathcal Z=\{Z_1,\ldots,Z_n\}$ of $\g$ such that $Z_l,\ldots, Z_n$ is a basis of $[\g,\g]$ for some $l\geq 1$.
Via the exponential map, we can identify $G$  with the vector space
$\g\simeq\mathbb R^n$ as manifolds and if we equip $\g$ with the Campbell-Baker-Hausdorff product $\cdot_\g$ then $\exp:(\g,\cdot_\g)\to (G,\cdot)$ is a group isomorphism.
By \cite[Theorem 6.1.13]{jean_book},  for every Jordan-H\"older basis
$\mathcal Z=\{Z_1,\ldots Z_n\}$ of $\g$, there exist polynomial functions $q_j$, $j=2,\ldots, n$, defined on $\mathbb R^{j-1}\times\mathbb R^{j-1}$, such that for any $s$, $t\in\mathbb R^n$, for $X=\sum_{j=1}^ns_jZ_j$, $Y=\sum_{j=1}^nt_jZ_j$ one has
$$X\cdot_\g Y=\sum_{j = 1}^n(s_j+t_j+ q_j(s_1,\ldots,s_{j-1},t_1,\ldots,t_{j-1}))Z_j,$$
where $q_1\equiv 0$.
On the other hand, due to the Campbell-Baker-Hausdorff formula, we have
$X\cdot_\g Y=X+Y+Z$, where $Z\in[\g,\g]$. Hence,
 if $Y\in[\g,\g]$, we have $t_j=0$, $j=1,\ldots, l-1$, and
\begin{eqnarray*}
X\cdot_\g Y
& = &
\sum_{i=1}^{l-1}s_jZ_j + \left(s_l+t_l+q_l(s_1,\ldots, s_{l-1})\right)Z_l\\
& + &
\sum_{j=l+1}^n(s_j+t_j+q_j(s_1,\ldots,s_{j-1},t_l,\ldots,t_{j-1}))Z_j.
\end{eqnarray*}
Let now $K$ be a proper compact subset of $G/[G,G]$.
The latter group can be identified with $\g/[\g,\g]\simeq \mathbb R^{l-1}$; we will also identify $[G,G]$ with $[\g,\g]\simeq \mathbb R^{n-l+1}$ ($\sum_{i=l}^ns_iZ_i\to (s_l,\ldots,s_n)\in\mathbb R^{n-l+1}$ with the multiplication induced by $\cdot_{\g}$).
Let $\theta_1 \in C_c^\infty(\mathbb R^{l-1})$ be
such that $\theta_1=1$ on a neighborhood of $K$ and $\theta_2\in C_c^\infty(\mathbb R^{n-l+1})$ such that
$\int_{\mathbb R^{n-l+1}}\theta_2(t) dt=1$.
Identifying $G$ with $\mathbb R^n$ we
set $\theta(x_1,\ldots, x_n)=\theta_1(x_1,\ldots,x_{l-1})\theta_2(x_{l},\ldots x_n)$.
As $\theta\in C_c^\infty(\mathbb R^n)$,
we have $\theta\in A(G)$ (see e.g. \cite[(3.23)]{eym} or \cite[(3.8) and Lemma 3.3]{lutu}).
Moreover, as the Haar measure on $[G,G]$, identified with
$\mathbb R^{n-l+1}$, is the Lebesgue measure, we have for $X=(s_1,\ldots,s_n)$
\begin{eqnarray*}
& & \phi(\theta)(\dot{X})
= \int_{[G,G]}\theta (X\cdot_\g Y)dY\\
& = & \int_{\mathbb R^{n-l+1}}\theta_1(s_1,\ldots,s_{l-1})\theta_2(s_{l}+t_{l}+q_{l}(s_1,\ldots,s_{l-1}),\ldots,\\&&s_n+t_n+q_n(s_1,\ldots,s_{n-1}, t_{l},\ldots, t_{n-1}))dt_{l}\ldots dt_n\\
& = & \theta_1(s_1,\ldots,s_{l-1})\int_{\mathbb R^{n-l+1}}\theta_2(t)dt=\theta_1(s_1,\ldots,s_{l-1}).
\end{eqnarray*}
\end{proof}

We finish this section with a more precise
form of the failure of reduced spectral synthesis, inspired by the results in \cite{collela}.

\begin{theorem}\label{p_clint}
Let $G$ be a non-discrete second countable locally compact group that fails reduced spectral synthesis.
Then there exists a subset of $G$ which is the closure of its interior and does not satisfy reduced spectral synthesis.
\end{theorem}
\begin{proof}
Let $E\subseteq G$ be a closed set that fails reduced spectral synthesis.
Let $\nph\in I(E)$ and $S\in C_r^*(G)$ be such that
$\supp_{\vn}(S) \subseteq E$ and $\langle S,\nph\rangle\ne 0$.
Recall  that the successive Cantor-Bendixson
derivatives of the set $E$ are defined as follows: let $E_0 = E$ and, for an ordinal
$\beta$, let $E_{\beta}$ be equal to the set of all
limit points of $E_{\beta-1}$ if $\beta$ has a predecessor, and to
$\cap_{\gamma < \beta} E_{\gamma}$ if $\beta$ is a limit ordinal.
Since $G$ is second countable, it is a Polish space. Hence, as $E$ is a closed subset of $G$, there exists a countable ordinal $\alpha$
such that $E_{\alpha} = E_{\alpha+1}$ and  $E_{\alpha}$ is a perfect set \cite[Theorem 6.11]{kechris}.
Moreover, $E_{\beta}\setminus E_{\beta+1}$ is a countable set consisting of
the isolated points of $E_\beta$ for every $\beta < \alpha$.

We claim that $E_\alpha$ fails reduced spectral synthesis.
It suffices to show that $\supp_{\vn}(S) \subseteq E_\alpha$.
We proceed by transfinite induction.
Let $\beta \leq \alpha$ and assume that $\supp_{\vn}(S) \subseteq E_{\gamma}$ for all $\gamma < \beta$.
Suppose that $\beta$ has a predecessor and let
$\{x\}$ be an isolated point of $E_{\beta-1}$.
By the regularity of $A(G)$,
there exists $u \in A(G)$ such that $\supp(u) \cap (E_{\beta-1}\setminus\{x\}) = \emptyset$ and
$u = 1$ on a compact neighborhood of $x$.
By \cite[Proposition 4.8]{eym},
$\supp_{\vn} (u \cdot S) \subseteq \{x\}$.
Note that $1-u\in B(G)$. Since $1-u=0$ on a neighbourhood of $x$,
for any $v \in J\left(E_{\beta-1}\setminus\{x\}\right)$, we have
that $(1-u)v \in J(E_{\beta-1})$  and hence
$\langle S-u\cdot S,v\rangle = \langle S,(1-u)v \rangle = 0$, implying
that $\supp_{\vn} (S -u\cdot S) \subseteq E_{\beta-1}\setminus\{x\}$.
By \cite[Corollary 5.3]{gralmul}, $\{x\}$ is a set of uniqueness and, since
$u \cdot S \in C_r^*(G)$, we have that $u \cdot S = 0$.
It follows that $\supp_{\vn}(S)\subseteq E_{\beta-1}\setminus\{x\}$.
Since this holds for any isolated point $x$ of $E_{\beta}$, we obtain
$\supp_{\vn}(S)\subseteq E_{\beta}$.
Now suppose that $\beta$ is a limit ordinal, and write $E_{\beta} = \cap_{n=1}^{\infty} E_{\beta_n}$,
for $\beta_n < \beta$, $n\in \bb{N}$, and $\beta = \cup_{n=1}^{\infty} \beta_n$.
Then $\supp_{\vn}(S)\subseteq E_{\beta_n}$ for all $n$, and hence $\supp_{\vn}(S)\subseteq E_{\beta}$.
We conclude that
$\supp_{\vn}(S)\subseteq E_\beta$ for all $\beta\leq\alpha$.
We hence hereafter assume that $E = E_{\alpha}$.

Let $\partial E$ be the boundary of $E$ and $(x_n)_{n\in \bb{N}}$ be a sequence,
dense in $\partial E$.
Let $y_n\not\in E$ be such that $\text{dist}(x_n,y_n)\to 0$ and $|\nph(y_n)|<2^{-2n}$.
Choose $\psi_n\in A(G)$  such that $\psi_n(y_n)=1$ and $\|\psi_n\|_{A(G)} = 1$.  Then
$\varphi(y_n)\psi_n - \varphi$ vanishes at $y_n$, $n\in \bb{N}$.
Since singletons satisfy spectral synthesis, there exists $\phi_n\in A(G)$
vanishing on an open neighborhood $I_n$ of $y_n$ and such that
$\|\phi_n-\varphi(y_n)\psi_n+\varphi\|_{A(G)} < 2^{-2n}$. Let $\tilde\varphi_n=\phi_n+\varphi$.
We have that
$\tilde\varphi_n = \varphi$ on $I_n$ and
$$\|\tilde\varphi_n\|_{A(G)}\leq 2^{-2n}+|\varphi(y_n)|\|\psi_n\|_{A(G)} \leq 2^{-n}.$$

We next show that
there exist a compactly supported function $\varphi_n\in A(G)$ and a  compact neighborhood $J_n$ of $y_n$, $n\in \bb{N}$,
such that
\begin{itemize}
\item[(i)] $\supp(\nph_n) \cap \supp(\nph_m) = \emptyset$ and $J_n\cap J_m = \emptyset$ if $n\neq m$,
\item[(ii)] $\|\varphi_n\|\leq 2^{-n+1}$, $\supp(\nph_n) \cap  E = \emptyset$ and $\varphi_n = \varphi$ on $J_n$ for all $n$.
\end{itemize}
To this end,
let $U_n$ be an open neighbourhood of $y_n$, $n\in \bb{N}$, such that
$U_n\cap U_m = \emptyset$, for $n\neq m$, and $U_n \cap E = \emptyset$, $n\in \bb{N}$.
We claim that there exist a constant $C > 0$ and functions
$u_n\in A(G)$ such that $\supp(u_n) \subseteq U_n$,
$u_n = 1$ on a compact neighbourhood $K_n$ of $y_n$, and $\|u_n\|\leq C$, $n\in \bb{N}$.
For a compact neighbourhood $M_n$ of $e$ with $M_n\subseteq y_n^{-1} U_n$,
let $V_n$ be a symmetric neighbourhood of $e$ such that $M_nV_n^2\subseteq y_n^{-1} U_n$.
Let $u_{M_n} = \chi_{M_nV_n}\ast\check\chi_{V_n}/m(V_n)$. Then
$u_{M_n}(x) = 1$ if $x\in M_n$, $u_{M_n}(x) = 0$ if $x\notin y_n^{-1} U_n$ and $\|u_{M_n}\|_{A(G)} \leq m(M_nV_n)^{1/2}/m(V_n)^{1/2}$.
As $G$ is second countable, there exists a decreasing sequence $(L_p)_{p\in \bb{N}}$ of
compact neighbourhoods of $e$ such that $\cap_{p\in \bb{N}} L_p = \{e\}$.
It follows that we may choose  the compact set $M_n\subseteq G$ such that $m(M_nV_n)^{1/2}/m(V_n)^{1/2} < 2$.
Now let $u_n$ be given by $u_n(x) = u_{M_n}(y_n^{-1} x)$, $x\in G$, and set
$\nph_n = \tilde\nph_n u_n$ and $K_n = y_n M_n$.
Let $J_n = \overline{I_n \cap \text{int}(K_n)}$.

Note that the function $\nph - \sum_{n=1}^{\infty} \nph_n$ belongs to $A(G)$ and vanishes on
the set $D \stackrel{def}{=} E\cup(\cup_{n\in \bb{N}} J_n)$.
Since $\nph_n$ has compact support, disjoint from $E$, and $\supp_{\vn}(S)\subseteq E$, we have that $\langle S, \nph_n\rangle = 0$, $n\in \bb{N}$.
Thus,
$$\left\langle S, \nph - \sum_{n=1}^{\infty} \nph_n\right\rangle = \langle S,\nph\rangle\ne 0$$
and hence $D$ is not a set of reduced spectral synthesis.
Finally, note that $D = \overline{\text{int}(D)}$.
\end{proof}


\section{Compact operator synthesis}\label{s_cos}

In this section, we define an operator version of reduced spectral synthesis,
exhibit a number of examples and establish some of its functorial properties.

\subsection{Definitions and basic properties}\label{s_dbp}
If $H_1$ and $H_2$ are Hilbert spaces, we denote by
$\cl B(H_1,H_2)$ the space of all bounded linear operators from $H_1$
into $H_2$, and by $\cl K(H_1,H_2)$ (resp. $\cl C_1(H_1,H_2)$,
$\cl C_2(H_1,H_2)$) the space of compact (resp. nuclear,
Hilbert-Schmidt) operators in $\cl B(H_1,H_2)$. We set $\cl K = \cl K(H_1,H_2)$ and $\cl C_2 = \cl C_2(H_1,H_2)$ for brevity.
As usual, we
write $\cl B(H) = \cl B(H,H)$. The space $\cl C_1(H_2,H_1)$ (resp.
$\cl B(H_1,H_2)$) can be naturally identified with the Banach space
dual of $\cl K$ (resp. $\cl C_1(H_2,H_1)$), the duality
being given by the map $(T,S)\to \langle T,S\rangle
\stackrel{def}{=}\tr(TS)$. Here $\tr A$ denotes the trace of a
nuclear operator $A$.

Let $(X,\mu)$ be a standard measure space; thus, $X$ is equipped with the Borel $\sigma$-algebra of
a locally compact Hausdorff metrisable complete topology, with respect to which $\mu$ becomes a regular Borel
$\sigma$-finite measure.
For a measurable set $\alpha\subseteq X$, we write $\chi_{\alpha}$ for its characteristic function.
If $a\in L^{\infty}(X,\mu)$, let $M_a\in \cl B(L^2(X,\mu))$ denote the
operator of multiplication by $a$. The collection
$$\cl D_X = \left\{M_a : a\in L^{\infty}(X,\mu)\right\}$$
is a maximal abelian selfadjoint
algebra (for short, masa) on $L^2(X,\mu)$.
If $\alpha\subseteq X$ is a measurable set, we let $P(\alpha) = M_{\chi_{\alpha}}$.

Let $(Y,\nu)$ be a(nother) standard measure space which, along with $(X,\mu)$, will be
fixed throughout this section.
We set $H_1 = L^2(X,\mu)$ and $H_2 = L^2(Y,\nu)$, and equip
$X\times Y$ with the product measure $\mu\times \nu$.
A subset of
$X\times Y$ is said to be a {\it measurable rectangle} (or simply a
{\it rectangle}) if it has the form $\alpha\times\beta$, where
$\alpha\subseteq X$ and $\beta\subseteq Y$ are measurable subsets. A
subset $E\subseteq X\times Y$ is called {\it marginally null} if
$E\subseteq (M\times Y)\cup(X\times N)$, where $M\subseteq X$ and $N\subseteq Y$ are null.
Two subsets $E,F\subseteq X\times Y$ are called
{\it marginally equivalent} (written $E\cong F$) if their symmetric difference is marginally null.
We say that $E$ is \emph{marginally contained} in $F$ if $E\setminus F$ is marginally null, and write
$E\subseteq_{\omega} F$.
A subset $E$ of $X\times Y$ is called {\it $\omega$-open} if it is
marginally equivalent to the union of a countable set of rectangles.
The complements of $\omega$-open sets are called {\it
$\omega$-closed}. It is clear that the class of all $\omega$-open
(resp. $\omega$-closed) sets is closed under countable unions (resp.
intersections) and finite intersections (resp. unions).
Given any family $\cl E$ of $\omega$-open sets, there exists a smallest, with respect to marginal inclusion,
$\omega$-open set $\cup_{\omega}\cl E$ with the property that $E\subseteq_{\omega} \cup_{\omega}\cl E$ for every $E\in \cl E$.
Let
$${\rm int}_{\omega}(\kappa) = \cup_{\omega}\left\{R : \mbox{ a rectangle with } R\subseteq_{\omega} \kappa\right\}$$
be the \emph{$\omega$-interior} of the measurable subset $\kappa\subseteq X\times Y$; thus,
${\rm int}_{\omega}(\kappa)$ is the largest, with respect to marginal inclusion,
$\omega$-open set contained in $\kappa$.
The \emph{$\omega$-closure} of $\kappa$ is the set ${\rm cl}_{\omega}(\kappa) = ({\rm int}_{\omega}(\kappa^c))^c$;
thus, ${\rm cl}_{\omega}(\kappa)$ is the smallest, with respect to marginal inclusion,
$\omega$-closed set containing $\kappa$.
We say that an $\omega$-closed set $\kappa$
is \emph{generated by rectangles} if $\kappa = {\rm cl}_{\omega}({\rm int}_{\omega}(\kappa))$.
We write $\partial_{\omega}(\kappa) \stackrel{def}{=} \kappa \setminus {\rm int}_{\omega}(\kappa)$
and call $\partial_{\omega}(\kappa)$ the \emph{$\omega$-boundary} of $\kappa$.

The space $L^2(X\times Y)$ will be identified with
$\cl C_2(H_1,H_2)$ via the map sending an element $k\in L^2(X\times Y)$
to the integral operator $T_k$ given by $T_k \xi (y) = \int_{X} k(x,y)\xi(x) d\mu(x)$,
$\xi\in H_1$, $y\in Y$.
On the other hand, the space $\cl C_1(H_2,H_1)$ can be canonically identified
with the Banach space projective tensor product $H_1\hat\otimes H_2$ and hence with the space
${\cl T}(X,Y)$ of all functions $h : X\times Y\to{\mathbb C}$, defined up to a marginally null set,
that admit a representation
$$h(x,y) = \sum_{i=1}^{\infty} f_i(x)g_i(y), \ \ \ (x,y)\in X\times Y,$$
where $(f_i)_{i\in \bb{N}}\subseteq H_1$ and $(g_i)_{i\in \bb{N}}\subseteq H_2$
are square summable sequences.
We set ${\cl T}(X) = {\cl T}(X,X)$.
Note that the duality between $\cl B(H_1,H_2)$ and ${\cl T}(X,Y)$ is given by
$$\langle T,f\otimes g\rangle = (Tf,\bar g), \ \ T\in \cl B(H_1,H_2), \ \ \ f\in H_1, g\in H_2,$$
where we write $f\otimes g$ for the function defined by $(f\otimes g)(x,y) = f(x) g(y)$.

Recall that a measurable essentially bounded function $\nph : X\times Y\to \bb{C}$ is called a
\emph{Schur multiplier} if
$\nph h$ is equivalent, with respect to product measure, to a function from
${\cl T}(X,Y)$ for every $h\in {\cl T}(X,Y)$. We write $\frak{S}(X,Y)$ for the space of all Schur multipliers.
If $\nph\in \frak{S}(X,Y)$ then the map $m_{\nph} : h\mapsto \nph h$ on ${\cl T}(X,Y)$ is automatically bounded, and
its adjoint, acting on $\cl B(H_1,H_2)$, is denoted by $S_{\nph}$.
A subspace $\cl W\subseteq \cl B(H_1,H_2)$ will
be called a \emph{masa-bimodule} if $M_{b}TM_{a} \in \cl  W$
for all $T\in \cl W$, $a\in L^{\infty}(X,\mu)$ and $b\in L^{\infty}(Y,\nu)$.
Note that a weak* closed subspace $\cl W \subseteq \cl B(H_1,H_2)$ is a masa-bimodule if and only if
$S_{\nph}(\cl W)\subseteq \cl W$ for every $\nph \in \frak{S}(X,Y)$.

We say that a measurable subset $\kappa\subseteq X\times Y$
{\it supports} an operator $T\in \cl B(H_1,H_2)$
(or that $T$ {\it is supported by} $\kappa$)
if $P(\beta)TP(\alpha) = 0$ whenever the rectangle $\alpha\times\beta$ is
marginally disjoint from $\kappa$, and write
$${\mathfrak M}_{\max}(\kappa) = \left\{T\in \cl B(H_1,H_2) : T \mbox{ is supported by } \kappa\right\}.$$
For any subset ${\mathcal W}\subseteq \cl B(H_1,H_2)$, there exists a smallest
(up to marginal equivalence) $\omega$-closed set
$\supp(\cl W)$ which supports every operator $T\in{\mathcal W}$ \cite{eks}.

By \cite{arv} and \cite{st1}, for any $\omega$-closed set $\kappa$, there exists a
weak* closed bimodule ${\mathfrak M}_{\min}(\kappa)$ such that, given a weak* closed bimodule
$\mathfrak M\subseteq \cl B(H_1,H_2)$, we have that
$\text{supp }{\mathfrak M} \cong \kappa$ if and only if
$${\mathfrak M}_{\min}(\kappa)\subseteq {\mathfrak M}\subseteq {\mathfrak M}_{\max}(\kappa).$$
We say that $\kappa$ is a set of {\it operator synthesis} if
${\mathfrak M}_{\min}(\kappa) = {\mathfrak M}_{\max}(\kappa)$ \cite{arv} (see also \cite{st1,tod}).

Let
$$\Phi(\kappa) = \{h\in {\cl T}(X,Y): h\chi_{\kappa} = 0 \mbox{ marginally a.e.}\}$$
and
$$\Psi(\kappa) = \{h\in {\cl T}(X,Y): h \text{ vanishes on an $\omega$-open nbhd of }\kappa\}^{-\|\cdot\|_{{\cl T}(X,Y)}}.$$
By \cite[Theorems 4.3 and 4.4]{st1},
\begin{equation}\label{eq_minmax}
{\mathfrak M}_{\min}(\kappa) = \Phi(\kappa)^{\perp} \ \mbox{ and } \ {\mathfrak M}_{\max}(\kappa)=\Psi(\kappa)^\perp.
\end{equation}

\begin{definition}\rm
An $\omega$-closed subset $\kappa\subseteq X\times Y$ will be called a set of {\it compact operator synthesis} if
$${\mathfrak M}_{\min}(\kappa)\cap{\cl K} = {\mathfrak M}_{\max}(\kappa)\cap{\cl K}.$$
\end{definition}

\noindent {\bf Remarks. (i)}
Every set of operator synthesis is trivially a set of compact operator synthesis.

\smallskip

{\bf (ii)}
Recall \cite{gralmul} that an $\omega$-closed set $\kappa\subseteq X\times Y$ is called an \emph{operator $U$-set}
if ${\mathfrak M}_{\max}(\kappa)\cap{\cl K} = \{0\}$. It is clear that if $\kappa$ is an operator $U$-set then it is a set of compact operator synthesis.
In fact, the following more general result holds:

\begin{proposition}\label{p_bo}
Let $\kappa\subseteq X\times Y$ be an $\omega$-closed set.
If $\partial_{\omega}(\kappa)$ is an operator $U$-set then $\kappa$ satisfies compact operator synthesis.
\end{proposition}

\begin{proof}
Let $T\in \mathfrak M_{\max}(\kappa)\cap {\cl K}$ and let $\nph : X\times Y \to \bb{C}$ be a Schur multiplier
that vanishes marginally almost everywhere on $\kappa$.
Then $S_\nph(T) \in \mathfrak M_{\max}(\kappa)\cap \cl K$.
Suppose that $\alpha\times\beta$ is a rectangle of finite product measure
with $\alpha\times\beta \subseteq_{\omega} \kappa$. Then, for every $\xi\in L^{\infty}(X)$ and $\eta\in L^{\infty}(Y)$, we have
$$\left\langle P(\beta) S_\nph(T) P(\alpha)\xi,\eta \right\rangle
= \left\langle T, \nph (\chi_{\alpha}\xi \otimes \chi_{\beta}\eta) \right\rangle = 0.$$
Thus, $S_\nph(T) \in \mathfrak M_{\max}({\rm cl}_{\omega}(\kappa^c))$.
It follows from the proof of \cite[Theorem 4.2]{eks} that
$S_{\nph}(T)$ is supported by the intersection of
$\kappa$ and ${\rm cl}_{\omega}(\kappa^c)$, that is, by
$\partial_{\omega}(\kappa)$.
Since $\partial_{\omega}(\kappa)$ is an operator  $U$-set, $S_\nph(T) = 0$.

By \cite{st2},
$$\mathfrak M_{\min}(\kappa) =
\{S : S_\theta(S)=0, \ \forall \ \theta\in \frak{S}(X,Y) \mbox{ such that } \theta\chi_\kappa = 0 \mbox{ m.a.e.}\}.$$
It follows that $T\in \mathfrak M_{\min}(\kappa)$ and hence $\kappa$ is a set of compact operator synthesis.
\end{proof}

We next show that compact operator synthesis is a concept of \lq\lq local'' nature.
If $\kappa\subseteq X\times Y$ is a measurable subset for which its characteristic function $\chi_{\kappa}$
is a Schur multipplier, we denote the idempotent $S_{\chi_{\kappa}}$ by $\pi_{\kappa}$.
A subset $Q\subseteq X\times Y$ will be called {\it elementary} if there exists a (finite)
family $\{\Pi_i\}_{i=1}^m$ of rectangles, such that
\begin{equation}\label{elSet}
Q = \cup_{i=1}^m \Pi_i.
\end{equation}
One may assume, after changing notation if necessary,
that the rectangles $\Pi_i$, appearing in (\ref{elSet}), are mutually disjoint.
It is clear that $\chi_Q$ is a Schur multiplier; in fact, writing
$\Pi_i = \alpha_i\times \beta_i$ for some measurable sets $\alpha_i\subseteq X$ and $\beta_i\subseteq Y$, $i = 1,\dots, m$,
we have that
$$\pi_Q(T) = \sum_{i=1}^m P(\beta_i) TP(\alpha_i), \ \ \ T\in \cl B(H_1,H_2).$$
Note that the range of $\pi_Q$ coincides with $\frak{M}_{\max}(Q)$.

\begin{lemma}\label{l_restrr}
Let $\kappa\subseteq X\times Y$ be an $\omega$-closed set,
$(\Pi_n)_{n\in \bb{N}}$ be a sequence of rectangles such that
$\Pi_n \subseteq_{\omega} \Pi_{n+1}$, $n\in \bb{N}$, and $\Pi = \cup_{n=1}^{\infty} \Pi_n$.
Suppose that $\kappa\cap\Pi_n$ is a set of compact
operator synthesis for each $n\in \bb{N}$. Then $\kappa\cap \Pi$ is a set of
compact operator synthesis.
\end{lemma}

\begin{proof}
Write
$\Pi_n = \alpha_n\times \beta_n$, for some measurable sets $\alpha_n\subseteq X$ and $\beta_n\subseteq Y$, $n\in \bb{N}$.
Setting $\alpha = \cup_{n=1}^{\infty}\alpha_n$ and $\beta = \cup_{n=1}^{\infty}\beta_n$, we have that $\Pi \cong \alpha\times \beta$.
Let $T\in \frak{M}_{\max}(\kappa\cap \Pi) \cap \cl K$. Then
$P(\beta_n)TP(\alpha_n)\in \frak{M}_{\max}(\kappa\cap \Pi_n) \cap \cl K$, and hence
$P(\beta_n)TP(\alpha_n) \in \frak{M}_{\min}(\kappa\cap \Pi_n)$, $n\in \bb{N}$.
Since $T = \lim_{n\to\infty} P(\beta_n)TP(\alpha_n)$ in norm, we conclude that $T\in \frak{M}_{\min}(\kappa)$.
\end{proof}

We will need the following reformulation of the \lq\lq principle of $\varepsilon$-compactness'' \cite[Lemma 3.4]{eks}.

\begin{lemma}\label{phi-zero}
Let $(X,\mu)$ and $(Y,\nu)$ be finite standard measure spaces,
$\kappa\subseteq X\times Y$ be an $\omega$-closed set, and $\gamma_n\subseteq X\times Y$ be $\omega$-open sets, $n\in \bb{N}$,
such that that $\kappa\subseteq \cup_{n=1}^{\infty} \gamma_n$.
For  each $\varepsilon > 0$, there exist $N\in \bb{N}$ and measurable subsets $\alpha \subseteq X$ and $\beta\subseteq Y$
with $\mu(\alpha) < \varepsilon$ and $\nu(\beta) < \varepsilon$ such that
\begin{equation}\label{eCom}
\kappa \subseteq  \left(\cup_{n=1}^{N} \gamma_n\right) \cup (\alpha\times Y) \cup (X\times \beta).
\end{equation}
\end{lemma}

\begin{proposition}\label{p_local}
Let $\kappa\subseteq X\times Y$ be an $\omega$-closed set and $\cl E$ be a family of $\omega$-open sets
such that $\kappa\subseteq \cup_{\omega}\cl E$ and $\kappa\cap E$ is a set of compact operator synthesis
for every $E\in \cl E$. Then $\kappa$ is a set of compact operator synthesis.
\end{proposition}

\begin{proof}
Suppose first that the measures $\mu$ and $\nu$ are finite.
By \cite[Lemma 2.1]{stt}, there exists a countable set $\{E_k\}_{k\in \bb{N}}\subseteq \cl E$ such that
$\cup_{\omega}\cl E \cong \cup_{k\in \bb{N}} E_k$.
Fix $\epsilon > 0$.
By Lemma \ref{phi-zero}, there exist measurable subsets
$X_{\epsilon}\subseteq X$ and $Y_{\epsilon}\subseteq Y$ such that
$\mu(X\setminus X_{\epsilon}) < \epsilon$, $\nu(Y\setminus Y_{\epsilon}) < \epsilon$,
positive integers $N$ and $m_k$, $k\in \bb{N}$,
and measurable sets $\alpha_{k,i}\subseteq X$, $\beta_{k,i}\subseteq Y$, $i = 1,\dots,m_k$,
such that $\cup_{i=1}^{m_k} \alpha_{k,i}\times \beta_{k,i}\subseteq E_k$, $k= 1,\dots,N$, and
\begin{equation}\label{eq_inclus}
\kappa\cap (X_{\epsilon}\times Y_{\epsilon}) \subseteq \cup_{k=1}^N \cup_{i=1}^{m_k} \alpha_{k,i}\times \beta_{k,i}.
\end{equation}
Let $\{R_j\}_{j \in J}$ be a (finite) family of mutually disjoint rectangles
such that each $R_j$ is contained in some $\alpha_{k,i}\times \beta_{k,i}$,
and
\begin{equation}\label{eq_eqall}
\cup_{k=1}^N \cup_{i=1}^{m_k} \alpha_{k,i}\times \beta_{k,i} = \cup_{j\in J} R_j.
\end{equation}
Let $T\in \frak{M}_{\max}(\kappa)\cap \cl K$ and $T_{\epsilon} = P(Y_{\epsilon}) T P(X_{\epsilon})$.
By (\ref{eq_inclus}) and (\ref{eq_eqall}), $T_{\epsilon}\in \frak{M}_{\max}(\cup_{j\in J} R_j)\cap \cl K$.
It follows that
\begin{equation}\label{eq_sumJ}
T_{\epsilon} = \sum_{j\in J} \pi_{R_j}(T_{\epsilon}).
\end{equation}
On the other hand,
$\pi_{R_j}(T_{\epsilon})\in \frak{M}_{\max}(\kappa \cap R_j)\cap \cl K$, $j\in J$.
By assumption, for each $j\in J$, there exists $k\in \bb{N}$ such that $R_j\subseteq E_k$.
Since $\kappa\cap E_k$ is a set of compact operator synthesis, we conclude that
$\pi_{R_j}(T_{\epsilon})\in \frak{M}_{\min}(\kappa \cap E_k)$.
By (\ref{eq_minmax}), the correspondence $\lambda\to \frak{M}_{\min}(\lambda)$ is monotone, and hence
$$\pi_{R_j}(T_{\epsilon})\in \frak{M}_{\min}(\kappa), \ \ \ j\in J.$$
Equation (\ref{eq_sumJ}) implies that $T_{\epsilon} \in \frak{M}_{\min}(\kappa)$.
Taking a limit as $\epsilon \to 0$, we obtain that $T \in \frak{M}_{\min}(\kappa)$.

Now relax the assumption that $\mu$ and $\nu$ be finite, and write $X = \cup_{n=1}^{\infty} X_n$ and
$Y = \cup_{n=1}^{\infty} Y_n$ as increasing unions of sets of finite measure.
By the previous paragraph, $\kappa\cap (X_n\times Y_n)$ is a set of compact operator synthesis for each $n\in \bb{N}$.
By Lemma \ref{l_restrr}, $\kappa$ is a set of compact operator synthesis.
\end{proof}

\noindent {\bf Remark. } We note that a statement, analogous to Proposition \ref{p_local},
where compact operator synthesis is replaced by operator synthesis,
also holds true, and can be obtained after straightforward modifications of the proof given above.

\medskip

The problem of when the compact operators in a certain operator space can be approximated by finite rank
operators in the same space has attracted substantial interest in non-selfadjoint operator algebra theory
(see e.g. \cite{dav-book} and \cite{ivan-luda-un}), and motivates the next result,
which provides a characterisation of the sets generated by rectangles that satisfy compact operator synthesis.
Let $\frak{F}$ be the set of all finite rank operators in $\cl B(H_1,H_2)$ and set
$${\mathfrak M}_{\rm f}(\kappa) = \frak{M}_{\max}(\kappa)\cap \frak{F}.$$

\begin{theorem}\label{approx}
Let $\kappa\subseteq X\times Y$ be an $\omega$-closed set.
If $\frak{M}_{\max}(\kappa)\cap \cl K =  \overline{\frak{M}_{\max}(\kappa)\cap \cl C_2}^{\|\cdot\|}$ then
$\kappa$ is a set of compact operator synthesis.
Thus, if
\begin{equation}\label{eq_fmax}
\overline{{\mathfrak M}_{\rm f}(\kappa)}^{\|\cdot\|} =  {\mathfrak M}_{\max}(\kappa)\cap\cl K
\end{equation}
then $\kappa$ is a set of
compact operator  synthesis.

Conversely,
if $\kappa\subseteq X\times Y$ is a set of compact operator synthesis that is generated by rectangles
then (\ref{eq_fmax}) holds true.
\end{theorem}

\begin{proof}
Let $k\in L^2(X\times Y)$ be such that $T_k\in \frak{M}_{\max}(\kappa)$.
For all subsets $X_0\subseteq X$ and $Y_0\subseteq Y$ of finite measure,
we have that $P(Y_0) T_k P(X_0) \in \frak{M}_{\max}((X_0\times Y_0)\cap \kappa)$.
By \cite[Lemma~6.1]{eks}, $k|_{X_0\times Y_0}$ vanishes almost everywhere
on $(X_0\times Y_0) \setminus \kappa$;
since $X$ and $Y$ are $\sigma$-finite, we have that $k$ vanishes almost everywhere outside $\kappa$.
Hence
$$\langle T_k,h\rangle = \int k(x,y) h(x,y)d\mu(x)d\nu(y) = 0, \ \ \ h\in \Phi(\kappa).$$
By (\ref{eq_minmax}), $T_k \in {\mathfrak M}_{\min}(\kappa)$.
Thus, if $\frak{M}_{\max}(\kappa)\cap \cl K =  \overline{\frak{M}_{\max}(\kappa)\cap \cl C_2}^{\|\cdot\|}$ and
$T\in \frak{M}_{\max}(\kappa)\cap \cl K$ then $T\in {\mathfrak M}_{\min}(\kappa)$.

Suppose that
$\kappa$ is the $\omega$-closure of its $\omega$-interior.
Since the support of a rank one operator is a rectangle,
$\supp\left({\mathfrak M}_{\rm f}(\kappa)\right)$ contains any rectangle contained in $\kappa$.
Thus,
$$\kappa = {\rm cl}_{\omega}({\rm int}_{\omega}(\kappa)) \subseteq
\supp\left({\mathfrak M}_{\rm f}(\kappa)\right) \subseteq \kappa,$$
and hence we have equality throughout.
On the other hand, ${\mathfrak M}_{\rm f}(\kappa)\subseteq \frak{M}_{\min}(\kappa)$
(see \cite[Theorem 3.6]{gralmul}).
By the minimality property of $\mathfrak M_{\min}(\kappa)$, we have
that $\overline{\frak{M}_{\rm f}(\kappa)}^{w^*} = \mathfrak M_{\min}(\kappa)$.

Now assume in addition  that $\kappa\subseteq X\times Y$ is a set of compact operator synthesis, and
let $T\in \frak{M}_{\max}(\kappa) \cap \cl K$.
Then $T\in \frak{M}_{\min}(\kappa)$ and hence
$T$ lies in the weak closure of $\frak{M}_{\rm f}(\kappa)$, when the latter is considered as a subspace
of $\cl K$.
The statement follows from the fact that the weak and the norm closures of any subspace of a Banach space
are identical.
\end{proof}

We finish this subsection with an \lq\lq inverse image theorem'' for compact operator synthesis.
The analogous result for operator synthesis was established in \cite[Theorem 4.7]{st1}.
Let $(X,\mu)$, $(X_1,\mu_1)$, $(Y,\nu)$ and $(Y_1,\nu_1)$ be standard measure spaces
and $\nph : X\rightarrow X_1$ and $\psi : Y\rightarrow Y_1$ be measurable mappings such that
the measure $\nph_*\mu$ on $X_1$, given by $\nph_*\mu(\alpha_1) = \mu(\nph^{-1}(\alpha_1))$,
is absolutely continuous with respect to $\mu_1$ and, similarly, the measure $\psi_*\nu$ is absolutely
continuous with respect to $\nu_1$.
We say that $\nph$ is one-to-one up to a null set if there exists a set $\Lambda\subseteq X$
such that $\mu(\Lambda) = 0$ and $\nph : \Lambda^c\to X_1$ is one-to-one.
Let $r : X_1\rightarrow \bb{R}^+$ be the Radon-Nikodym derivative of $\nph_*\mu$
with respect to $\mu_1$, that is, $r$ is a $\mu_1$-measurable function such that
$$\mu(\nph^{-1}(\alpha_1)) = \int_{\alpha_1} r(x_1)d\mu_1(x_1),$$ for every measurable
set $\alpha_1\subseteq X_1$.
Similarly, let $s : Y_1\rightarrow \bb{R}^+$
be the Radon-Nikodym derivative of $\psi_*\nu$ with respect to $\nu_1$.
Let $M_1 = \{x_1\in X_1 : r(x_1) = 0\}$ (resp. $N_1  = \{y_1\in Y_1 : s(y_1) = 0\}$).
Note that
$$\nph_*\mu(M_1) = \int_{M_1} r(x_1)d\mu_1(x_1) = 0.$$
Similarly, $\psi_*\nu(N_1) = 0$.

Define an operator $V_{\nph} : L^2(X_1,\mu_1)\rightarrow L^2(X,\mu)$ by letting
$$V_{\nph}\xi (x) =
\begin{cases}
\frac{\xi(\nph(x))}{\sqrt{r(\nph(x))}} \;&\text{if }x\not\in \nph^{-1}(M_1)\\
0 &\text{if }x\in \nph^{-1}(M_1).
\end{cases}$$
As demonstrated in \cite[Lemma 5.4]{gralmul}, $V_{\nph}$ is a partial isometry with initial space $L^2(M_1^c,\mu_1|_{M_1^c})$.

\begin{theorem}\label{pr_invim}
Let $\kappa_1\subseteq X_1\times Y_1$ be an $\omega$-closed set and
$$\kappa = \{(x,y)\in X\times Y : (\nph(x),\psi(y))\in \kappa_1\}.$$
Assume that $X = \cup_{i=0}^N \tilde{X}_i$ and  $Y=\cup_{i=0}^M \tilde{Y}_i$ are partitions of $X$ and $Y$
into measurable sets, such that
\begin{itemize}
\item[(i)] $\nph|_{\tilde{X}_0}$ and $\psi|_{\tilde{Y}_0}$ are one-to-one, and
\item[(ii)] $\nph|_{\tilde{X}_i}$ and $\psi|_{\tilde{Y}_j}$ are constant almost everywhere for $i\geq 1$ and
$j \geq 1$.
\end{itemize}
If $\kappa_1$ is a set of compact operator synthesis then so is $\kappa$.
\end{theorem}

\begin{proof}
For a standard measure space $(Z,\sigma)$, we identify the spatial weak* tensor product
$\cl B(\ell^2)\bar\otimes L^\infty(Z,\sigma)$ with the algebra of all bounded functions
$F : Z\to \cl B(\ell^2)$ such that, for each $\xi\in \ell^2$, the functions $z\mapsto F(z)\xi$ and $z\mapsto F(z)^*\xi$ are measurable
\cite[Chapter IV, Section 7]{takesaki1}.
A pair $(P,Q)$, where $P\in \cl B(\ell^2)\bar\otimes L^\infty(X,\mu)$ and
$Q\in \cl B(\ell^2)\bar\otimes L^\infty(Y,\nu)$,
is called a $\kappa$-pair \cite{st1} if $P(x)Q(y) = 0$ marginally almost everywhere on $\kappa$.
One defines $\kappa_1$-pairs analogously.
Let $K \in \frak{M}_{\max}(\kappa) \cap \cl K$.

We first assume that
$\nph$ and $\psi$ are one-to-one up to a null set.
By \cite[Lemma 5.4]{gralmul}, $V_\nph$ and $V_\psi$ are surjective. Thus, if $K_1 = V_\psi^*KV_\nph$ then
$K = V_\psi K_1V_\nph^*$.
By the proof of \cite[Theorem 5.5]{gralmul},
the operator $K_1$ is compact and supported by $\kappa_1$.
Since $\kappa_1$ is a set of compact operator synthesis, we have $K_1\in {\mathfrak M}_{\min}(\kappa_1)$ and therefore,
by \cite[Corollary~4.4]{st1}, $Q_1(I\otimes K_1)P_1 = 0$ for any $\kappa_1$-pair $(P_1,Q_1)$.

Let $(P,Q)$ be a $\kappa$-pair.
By \cite{st1} and \cite[p. 1488]{gralmul},
there exists a $\kappa_1$-pair $(\hat{P}_1,\hat{Q}_1)$ such that
$P(x)\leq \hat{P}_1(\nph(x))$ (resp. $Q(y)\leq \hat{Q}_1(\psi(y))$) for almost all $x\in X$ (resp. $y\in Y$).
We have \begin{eqnarray*}
Q(I \otimes K)P
& = &
Q(I \otimes V_\psi K_1V_\nph^*)P =Q (\hat Q_1\circ\psi) (I \otimes V_\psi K_1V_\nph^*)(\hat P_1\circ\nph)P\\
& = &
Q(I\otimes V_\psi)\hat Q_1(I \otimes K_1)\hat P_1(I\otimes V_\nph^*)P.
\end{eqnarray*}
Since $(\hat P_1,\hat Q_1)$ is a $\kappa_1$-pair, we have that $\hat Q_1(I \otimes K_1)\hat P_1 = 0$. It follows that
$Q(I\otimes K)P = 0$
and therefore $K\in {\mathfrak M}_{\min}(\kappa)$.

For the general case,
let $F\in {\cl T}(X,Y)$ vanish marginally almost everywhere on $\kappa$ and set $K_i = KP(\tilde{X}_i)$, $i = 0,\dots,N$.
We may assume that there are distinct points $a_1,\dots,a_N \in X_1$ such that
$\varphi(x) = a_i$, $x\in \tilde{X}_i$, $i = 1,\dots,N$.
We have that
$\langle K,F\rangle = \sum_{i=0}^N \langle K_i, F\chi_{\tilde{X}_i\times Y}\rangle$.
Set
$\hat{Y}_i = \{y\in Y: (a_i,\psi(y))\in \kappa_1\}$, $i = 1,\dots,N$.
As $K$ is supported by $\kappa$, we have that
$K_i = M_{\chi_{\hat{Y}_i}}K_i$, $i = 1,\dots,N$, and
\begin{equation}\label{eq_Ki}
\left\langle K_i,F\right\rangle = \left\langle K_i,F \chi_{\tilde{X}_i \times \hat{Y}_i} \right\rangle.
\end{equation}
Since $\tilde{X}_i\times \hat{Y}_i\subseteq \kappa$, we have that
$F \chi_{\tilde{X}_i \times \hat{Y}_i} = 0$ marginally almost everywhere; by (\ref{eq_Ki}),
$$\langle K_i,F\rangle = 0, \ \ \ \ i = 1,\dots,N.$$

We may thus assume that $\tilde{X}_0 = X$.
Set $K_{0,j} = M_{\chi_{\tilde{Y}_j}}K$, $j = 0,\dots,M$.
By considering the operators $K^*$ and $K_{0,j}^*$, $j = 1,\dots,M$,
the previous paragraph implies that
$$\langle K_{0,j},F\rangle = 0, \ \ \ \ j = 1,\dots,M.$$
We thus may assume that $X = \tilde{X}_0$ and $Y = \tilde{Y}_0$.
The claim now follows from the first part of the proof.
\end{proof}

\subsection{Thin sets}\label{ss_ts}

In classical Harmonic Analysis, substantial attention has been devoted to the study of
various conditions of thinness of sets and their relation with spectral synthesis.
In this and the next subsections, we pursue a similar line of investigation in the operator setting,
identifying special classes of $\omega$-closed sets satisfying compact operator synthesis.

\begin{definition}\label{d_thin}
Let $\kappa\subseteq X\times Y$ be an $\omega$-closed set. We say that

(i) $\kappa$ is \emph{thin} if there is a decreasing sequence
$(Q_n)_{n=1}^{\infty}$ of elementary sets containing $\kappa$,
such that $\|\pi_{Q_n}(T)\| \to_{n\to\infty} 0$ for every compact operator $T$;

(ii) $\kappa$ \emph{has thin boundary} if $\kappa = \Gamma\cup E$, where $\Gamma$ is thin and $E$ is $\omega$-open.
\end{definition}

\begin{remark}\label{r_four}
{\rm
{\bf (i)} \ \
Let $\kappa$ be a thin set, and $(Q_n)_{n=1}^{\infty}$ be an associated sequence of elementary sets
as in Definition \ref{d_thin} (i). If $T\in \frak{M}_{\max}(\kappa)\cap \cl K$
then $\pi_{Q_n}(T) = T$ for each $n$, and hence $T = 0$. It follows that $\kappa$ is an
operator set of uniqueness, and therefore a set of compact operator synthesis.

\smallskip

{\bf (ii)} \
A subset of a thin set is thin. In particular, the intersection of finitely many thin sets is thin.
The same holds true for unions.
Indeed, suppose that a subset $\kappa\subseteq X\times Y$ is thin,
and let $(Q_n)_{n=1}^{\infty}$ be a decreasing  sequence
of elementary sets, containing $\kappa$, and such that $\|\pi_{Q_n}(T)\|\to_{n\to\infty} 0$ for every $T \in \cl K$.
As $\|\pi_{Q_n}|_{\cl K}\|=\|\pi_{Q_n}\|$, by the Banach-Steinhaus Theorem, $\sup_{n\in \bb{N}} \|\pi_{Q_n}\| < \infty$.
If $\kappa' \subseteq X\times Y$ is another thin set
and $(Q_n')_{n=1}^{\infty}$ is its corresponding sequence of elementary sets as in Definition \ref{d_thin} (i)
then, setting $R_n = Q_n\cup Q_n'$, $n\in \bb{N}$,
we have that
$R_{n+1}\subseteq R_n$,
$\kappa\cup \kappa' \subseteq R_n$, $n\in \bb{N}$,
and
$$\|\pi_{R_n}(T)\| = \|\pi_{Q_n}(T) + (\id-\pi_{Q_n})\pi_{Q_n'}(T)\|\to_{n\to\infty} 0$$
for every $T \in \cl K$.

\smallskip

{\bf (iii)} \
By (ii), the union (resp. intersection) of finitely many sets with thin boundary is a set with thin boundary.

\smallskip

{\bf (iv)}  An $\omega$-closed set $\kappa$ has thin boundary if and only if
$\partial_{\omega}(\kappa)$ is thin.
Indeed, suppose that
$\kappa = \Gamma\cup E$, where $\Gamma$ is thin and $E$ is $\omega$-open. Then $E\subseteq \kappa$
and hence $E\subseteq {\rm int}_{\omega}(\kappa)$. Thus, $\partial_{\omega}(\kappa) \subseteq \Gamma$
and so $\partial_{\omega}(\kappa)$ is thin.
}
\end{remark}

Since every thin set is an operator $U$-set, Theorem \ref{p_bo} implies that
any set with  thin boundary is a set of compact operator synthesis.
The following theorem gives a much  stronger statement.

\begin{theorem}\label{intersection}
Let $\tau\subseteq X\times Y$ be a set of compact operator synthesis and $\kappa$ be a set with thin boundary.
Then $\tau\cap \kappa$ and $\tau\cup \kappa$ are sets of compact operator synthesis.
\end{theorem}

\begin{proof}
By Lemma \ref{l_restrr}, it suffices to consider the case where the measures $\mu$ and $\nu$ are finite.
Write $\kappa = \Gamma\cup E$, where $\Gamma$ is thin and $E$ is $\omega$-open.
Let $D$ be an elementary set such that $\Gamma\subseteq D$.
The set $\kappa\setminus D$ coincides with $E\setminus D$ and is hence
$\omega$-open. By Lemma \ref{phi-zero},
there exist an elementary subset $E_0$ of $E\setminus D$ and subsets
$\alpha_{\epsilon} \subseteq X$ and $\beta_{\epsilon} \subseteq Y$
with $\mu(\alpha_{\epsilon}) < \varepsilon$ and $\nu(\beta_{\epsilon}) < \varepsilon$, such that
$$\kappa\setminus D \subseteq (\alpha_{\epsilon}\times Y) \cup (X\times \beta_{\epsilon}) \cup E_0.$$
Set $D_{\varepsilon} = D\cup (\alpha_{\epsilon}\times Y) \cup (X\times \beta_{\epsilon})$ and let
$T\in{\mathfrak M}_{\max}(\tau\cap \kappa) \cap \cl K$.  Then
$T\in {\mathfrak M}_{\max}(\tau)\cap\cl K$ and, since $\tau$ satisfies compact operator synthesis,
$T\in {\mathfrak M}_{\min}(\tau)$.
It follows from \cite[Lemma 3.6]{gralmul} that if $\Pi$ is a rectangle with $\Pi\subseteq \kappa$
then $\pi_{\Pi}(T)\in \frak{M}_{\min}(\tau\cap \Pi)$ and therefore
$\pi_{\Pi}(T)\in {\mathfrak M}_{\min}(\tau\cap \kappa)$. Thus,
\begin{equation}\label{int}
\pi_{\Pi}(T)\in {\mathfrak M}_{\min}(\tau\cap \kappa) \text {   for every elementary set } \Pi\subseteq \kappa.
\end{equation}
Set $Q = \kappa\cap D_{\varepsilon}^c$; the set $Q$ coincides with $E_0\cap D_{\varepsilon}^c$
and is hence elementary.
Since $T$ is supported by $\kappa$, we have
$$T = \pi_{D_{\varepsilon}}(T) +  \pi_{Q}(T).$$
Suppose that $F\in {\cl T}(X,Y)$  annihilates ${\mathfrak M}_{\min}(\tau\cap \kappa)\cap\cl K$.
By (\ref{int}),
$$\langle T,F\rangle = \langle\pi_{D_{\varepsilon}}(T),F\rangle+\langle\pi_{Q}(T),F\rangle = \langle\pi_{D_{\varepsilon}(T)},F\rangle.$$
Letting $\varepsilon\to 0$ we obtain $\langle T,F\rangle=\langle \pi_D(T),F\rangle$.
Since $\Gamma$ is thin, $D$ can be chosen so that $\|\pi_{D}(T)\|$ is arbitrarily small;
thus $\langle T,F\rangle = 0$ and hence $T\in \mathfrak M_{\min}(\tau\cap \kappa)$.

Let now $T\in{\mathfrak M}_{\max}(\tau\cup \kappa)\cap\cl K$
and suppose that $F\in \cl T(X,Y)$  annihilates
${\mathfrak M}_{\min}(\tau\cup \kappa)\cap\cl K$.
Set $S = T - \pi_{D_{\varepsilon}}(T)$.
We have
$$\langle T,F\rangle = \langle \pi_{D_{\varepsilon}}(T),F\rangle + \langle S,F\rangle.$$
The support of the operator $S$ is
a subset of $(\tau\cup \kappa)\cap D_{\epsilon}^c$; since
$\kappa\setminus D_{\epsilon}$ is contained in $E_0$, it is
contained in $\tau\cup E_0$.

For each rectangle $\Pi \subseteq E_0 $, we have
$$\pi_{\Pi}(S) \in \mathfrak M_{\max}(\Pi) \cap \cl K = \mathfrak M_{\min}(\Pi)  \cap \cl K
\subseteq \mathfrak M_{\min}(\kappa)  \cap \cl K,$$
whence $\langle \pi_{\Pi}(S),F\rangle = 0$.
Therefore
\begin{equation}\label{eq_e0s}
\langle \pi_{E_0}(S),F\rangle = 0.
\end{equation}
On the other hand,
$S-\pi_{E_0}(S)$ is supported by $\tau$; so
$$S - \pi_{E_0}(S) \in \mathfrak M_{\max}(\tau) \cap \cl K = \mathfrak M_{\min}(\tau)  \cap \cl K$$ and hence
$\langle (S- \pi_{E_0}(S)),F\rangle = 0$. By (\ref{eq_e0s}),
$\langle S,F\rangle = 0$, and so
$$\langle T,F\rangle = \langle \pi_{D_{\varepsilon}}(T),F\rangle.$$
Taking the limit as $\varepsilon\to 0$ we otain
$\langle T,F\rangle = \langle \pi_{D}(T),F\rangle,$
for each elementary set $D$ containing $\Gamma$. Since $\Gamma$ is thin, $\langle T,F\rangle =  0$.
\end{proof}

In connection with the study of synthetic properties of masa-bimodules,
the class of \emph{approximately $\frak{I}$-injective } masa-bimodules  was introduced in \cite{eletod}.
These are  weak* closed masa-bimodules $\cl U\subseteq \cl B(H_1,$ $H_2)$
for which  there exists a sequence
$(\kappa_n)_{n\in \bb{N}}$ of measurable subsets of $X\times Y$ such that
$\chi_{\kappa_n}$ is a Schur multiplier for each $n$,
the sequence $(\|\pi_{\kappa_n}\|)_{n\in \bb{N}}$ is bounded,
and $\cl U = \cap_{n\in \bb{N}} {\rm Ran} \pi_{\kappa_n}$.
Under these circumstances, the sets $\cap_{n=1}^{\infty}\kappa_n$ will be called
\emph{approximately $\frak{I}$-injective}.
It is straightforward to see that these sets are precisely the supports of approximately $\frak{I}$-injective masa-bimodules.
The connection between thin sets and approximately $\frak{I}$-injective masa-bimodules
is clarified in the next theorem.

\begin{theorem}\label{th_aithin}
Let $E\subseteq X\times Y$ be an $\omega$-closed set. The following are equivalent:

(i) \ $E$ is a thin set;

(ii) there exists an approximately $\frak{I}$-injective operator $U$-set $\kappa$ such that $E\subseteq \kappa$.
\end{theorem}

\begin{proof}
(i)$\Rightarrow$(ii) Suppose that $(Q_n)_{n\in \bb{N}}$ is a decreasing sequence of elementary sets
such that $E\subseteq \cap_{n=1}^{\infty} Q_n$ and $\pi_{Q_n}(T)\to 0$ for every $T\in \cl K$.
By the remarks before Lemma \ref{l_restrr}, $\chi_{Q_n}$ is a Schur multuplier.
The Uniform Boundedness Principle implies that the sequence
$(\pi_{Q_n})_{n\in \bb{N}}$ is bounded in norm. Thus, $\kappa = \cap_{n=1}^{\infty} Q_n$
is an approximately $\frak{I}$-injective set containing $E$.
If $T\in \frak{M}_{\max}(\kappa)\cap \cl K$ then
$\pi_{Q_n}(T) = T$ for every $n\in \bb{N}$, and hence $T = 0$. Thus, $\kappa$ is an operator $U$-set.

(ii)$\Rightarrow$(i)
Write $X=\cup_{n=1}^\infty X_n$ and $Y=\cup_{n=1}^\infty Y_n$,
where $(X_n)_{n\in \mathbb N}$ and $(Y_n)_{n\in \mathbb N}$ are increasing sequences of sets of finite measure.
Let $(\kappa_n)_{n\in \bb{N}}$ be a decreasing sequence
such that $\chi_{\kappa_n}$ is a Schur multiplier for each $n$, the sequence
$(\pi_{\kappa_n})_{n\in \bb{N}}$ is bounded in norm, and $\kappa = \cap_{n=1}^{\infty}\kappa_n$.
It follows easily from \cite[Theorem 6.5]{eks} that the pre-image of any open set under a Schur multiplier is $\omega$-open.
Thus, the set $\kappa_n$ is both $\omega$-closed and $\omega$-open.
By Lemma \ref{phi-zero}, there exist measurable sets $\alpha_n\subseteq X_n$ and $\beta_n\subseteq Y_n$
such that $\mu(\alpha_n) < \frac{1}{2^{n+2}}$, $\nu(\beta_n) < \frac{1}{2^{n+2}}$ and
the set
$\kappa_n \cap (\alpha_n^c\times \beta_n^c) \cap (X_n \times Y_n)$ is elementary.
Write $A_n = \cup_{k\geq n}\alpha_k$ and $B_n = \cup_{k\geq n}\beta_k$, $n\in \bb{N}$.
Set
$R_n = (A_n\times Y) \cup (X\times B_n);$
clearly, $R_n$, and thus $R_n^c$, is elementary.
Let
$$Q_n = (\kappa_n \cap R_n^c\cap (X_n\times Y_n)) \cup R_n \cup(X_n\times Y_n)^c, \ \ \ n\in \bb{N}.$$
Since
$$\kappa_n \cap R_n^c\cap (X_n\times Y_n) = \kappa_n \cap (\alpha_n^c \times \beta_n^c)\cap (X_n\times Y_n)\cap R_{n+1}^c,$$
we have that $Q_n$ is elementary.

We claim that
\begin{equation}\label{eq_inq}
Q_{n+1}\subseteq Q_n, \ \ \ n\in \bb{N}.
\end{equation}
To see this, note that $R_{n+1}\subseteq R_n$ and $(X_{n+1}\times Y_{n+1})^c\subseteq (X_{n}\times Y_{n})^c$.
Suppose that
$$(x,y)\in \kappa_{n+1} \cap R_{n+1}^c\cap (X_n\times Y_n).$$
If $(x,y)\in (\alpha_n\times Y_n) \cup (X_n\times \beta_n)$ then $(x,y)\in R_n$ and hence $(x,y)\in Q_n$.
If, on the other hand, $(x,y) \not\in (\alpha_n\times Y_n) \cup (X_n\times \beta_n)$ then, since
$(x,y)\in R_{n+1}^c$, we have that $(x,y)\in R_n^c$. Since $\kappa_{n+1}\subseteq \kappa_n$, we conclude that
$(x,y)\in \kappa_n\cap R_n^c\cap (X_n\times Y_n)$ and so $(x,y)\in Q_n$.
Hence
$$Q_{n+1}\cap (X_n\times Y_n)\subseteq Q_n.$$
On the other hand, $Q_{n+1}\cap (X_n\times Y_n)^c\subseteq (X_n\times Y_n)^c\subseteq Q_n$.
Inclusion (\ref{eq_inq}) follows.

We claim that
\begin{equation}\label{eq_inqn}
\cap_{n=1}^{\infty} Q_n \cong \kappa.
\end{equation}
First note that, by construction, $\kappa\subseteq Q_n$ for each $n$.
Set $M = \cap_{n=1}^\infty A_n$ and $N=\cap_{n=1}^\infty B_n$;
observe that
$\mu(M) = \lim_{n\to\infty}\mu(A_n) = 0$ and, similarly, $\nu(N) = 0$.
Since
$$\cap_{n=1}^\infty R_n = (M\times Y) \cup (X\times N),$$
we have that
$\cap_{n=1}^{\infty} R_n$ is marginally null.

Suppose that $(x,y)\in \left(\cap_{n=1}^{\infty} Q_n\right) \cap \left(M^c\times N^c\right)$. Then $(x,y)$ belongs to
infinitely many elements of the sequence
$$\left(\kappa_n\cap R_n^c\cap (X_n\times Y_n)\right)_{n\in \bb{N}}.$$
Indeed, otherwise it belongs to infinitely many
elements of the decreasing sequence
$$\left(R_n\cup (X_n\times Y_n)^c\right)_{n\in \bb{N}},$$
that is,
$$(x,y)\in \cap_{n=1}^\infty (R_n\cup (X_n\times Y_n)^c).$$
Since $\cap_{n=1}^\infty (X_n\times Y_n)^c = \emptyset$, this implies that
$(x,y)\in \cap_{n=1}^\infty R_n = (M\times Y) \cup (X\times N)$, a contradiction.
It follows that $(x,y)\in \cap_{n=1}^{\infty} \kappa_n$, that is, $(x,y)\in \kappa$, establishing (\ref{eq_inqn}).

Note that $R_n^c \cap (X_n\times Y_n) = (A_n^c\cap X_n)\times (B_n^c\cap Y_n)$, and hence
$\|\pi_{R_n^c \cap (X_n\times Y_n)}\|\leq 1$.
Thus,
\begin{eqnarray*}
\|\pi_{Q_n}\|
& \leq &
\|\pi_{\kappa_n\cap R_n^c \cap (X_n\times Y_n)}\| + \|\pi_{R_n\cup (X_n\times Y_n)^c}\|\\
& = &
\|\pi_{\kappa_n} \pi_{R_n^c \cap (X_n\times Y_n)}\| + \|\pi_{R_n\cup (X_n\times Y_n)^c}\|\\
& \leq &
\|\pi_{\kappa_n}\| \|\pi_{R_n^c \cap (X_n\times Y_n)}\| + \|\id - \pi_{R_n^c \cap (X_n\times Y_n)}\|
\leq
\|\pi_{\kappa_n}\| + 2.
\end{eqnarray*}
Let $T\in \cl K$. By \cite[Lemma 5.1 (ii)]{eletod}, the sequence
$(\pi_{Q_n}(T))_{n\in \bb{N}}$ converges in the operator norm.
As $\kappa\simeq\cap_{n=1}^\infty Q_n$,
its limit $S$ is  an element of $\frak{M}_{\max}(\kappa)\cap \cl K$.
Since $\kappa$ is an operator $U$-set, $S = 0$.
It follows that $\kappa$, and hence $E$, is a thin set.
\end{proof}

\subsection{Quasi-diagonal sets}\label{ss_qds}

In this subsection, we consider a special class of thin sets, which are defined in terms of limiting combinatorial
behaviour of covering elementary sets.
Let $Q$ be an elementary set. We call a (finite) family $D = \{\alpha_i\times\beta_i\}_{i=1}^n$
of rectangles a \emph{decomposition} of $Q$ if its elements are mutually disjoint and
$Q \cong \cup_{i=1}^n \alpha_i\times \beta_i$.
We let $N(D)$ be the smallest positive integer $N$ such that
$${\rm card}\{i : x\in \alpha_i\} \leq N \mbox{ and } {\rm card}\{i : y\in \beta_i\} \leq N,$$
for almost all $x\in X$ and almost all $y\in Y$
(here ${\rm card} J$ denotes the number of elements of a finite set $J$).
It is easy to see that
$$N(D) = \max\left\{\left\|\sum\chi_{\alpha_i}\right\|_\infty, \left\|\sum\chi_{\beta_i}\right\|_\infty\right\}.$$
We set
$$|Q| = \min_D N(D),$$
where the infimum is taken over all decompositions $D$ of $Q$.
For each decomposition
$D = \{\alpha_i\times\beta_i\}_{i=1}^n$ of $Q$,
we have
$\pi_Q(T) = \sum_{i=1}^n P_{\beta_i}TP_{\alpha_i}$; therefore,
\begin{eqnarray*}
\|\pi_Q\|
& \le &
\left\|\sum P_{\beta_i}P_{\beta_i}^{\ast}\right\|^{1/2} \left\|\sum P_{\alpha_i}^{\ast}P_{\alpha_i}\right\|^{1/2}\\
&=& \left\|\sum\chi_{\alpha_i}\right\|_\infty^{1/2} \left\|\sum\chi_{\beta_i}\right\|_{\infty}^{1/2}
\leq N(D).
\end{eqnarray*}
Taking the minimum over all decompositions $D$, we obtain
\begin{equation}\label{eq_inmd}
\|\pi_Q\| \le |Q|.
\end{equation}
Note that, by \cite{livshits}, the norm of the Schur multiplication by the matrix
$\left(\smallmatrix 1 & 1\\ 0 & 1\endsmallmatrix\right)$ is $\frac{2}{\sqrt{3}}$, showing that the inequality in
(\ref{eq_inmd}) may be strict.

For a rectangle $\Pi = \alpha\times \beta$, write $r(\Pi) = \min\{\mu(\alpha),\nu(\beta)\}$.
If $Q \subseteq X\times Y$ is an elementary set, define
$$r(Q) = \sup\left\{r(\Pi): \Pi\subseteq Q\right\};$$
for an arbitrary measurable subset $E \subseteq X\times Y$, let
$$r(E) = \inf \left\{r(Q): E \subseteq Q, \ Q \text{  elementary}\right\}.$$
Note that the map   $E\mapsto r(E)$ is monotone with respect to inclusion of elementary sets and therefore the
quantity $r(E)$ is well-defined.

\begin{definition}\label{d_qd}
A measurable subset $\kappa\subseteq X\times Y$ will be called \emph{quasi-diagonal}
(resp. \emph{strongly quasi-diagonal}) if
there exists a decreasing sequence $(Q_n)_{n\in \bb{N}}$ of elementary sets such that
\begin{itemize}
\item[(i)]
$\sup_{n\in \bb{N}} |Q_n| < \infty$,
\item[(ii)] $r(Q_n) \to_{n\to\infty} 0$, and
\item[(iii)] $\kappa \subseteq_{\omega} \cap_{n=1}^{\infty} Q_n$ (resp. $\kappa \cong\cap_{n=1}^{\infty} Q_n$).
\end{itemize}
\end{definition}

Special classes of
quasi-diagonal sets were considered in \cite[Section 6]{gralmul}
in connection with operator $U$-sets.

\begin{theorem}\label{newquasi}
Every quasi-diagonal set is thin.
\end{theorem}

\begin{proof}
Let $T$ be a rank one operator of the form $vu^*$, where $u$ (resp. $v$) is a bounded measurable
function on $X$ (resp. $Y$), and $vu^*$ is given by $vu^*(\xi) = (\xi,u)v$, $\xi\in H_1$.
Suppose that $D = \{\alpha_i\times \beta_i\}_{i=1}^m$ is a decomposition of
an elementary set $Q$.
Then
\begin{eqnarray*}
& & \|\pi_Q(T)\|
 \leq
\|\pi_Q(T)\|_2
= \left\|\sum_{i=1}^m(\chi_{\alpha_i}\otimes\chi_{\beta_i})(u\otimes \overline{v})\right\|_2\\
& \leq & \|u\|_\infty\|v\|_\infty\left(\sum_{i=1}^m \mu(\alpha_i)\nu(\beta_i)\right)^{1/2}\\
& \leq & \|u\|_\infty\|v\|_\infty\left(\sum_{i=1}^m r(\alpha_i\times \beta_i)(\mu(\alpha_i) + \nu(\beta_i))\right)^{1/2}\\
& \leq & \|u\|_\infty\|v\|_\infty r(Q)^{1/2} \left(\sum_{i=1}^m(\mu(\alpha_j)+\nu(\beta_j))\right)^{1/2}\\
& \leq & \|u\|_\infty\|v\|_\infty r(Q)^{1/2}
\left(\left\|\sum_{i=1}^m \chi_{\alpha_i}\right\|_\infty \mu(X) + \left\|\sum_{i=1}^m \chi_{\beta_i}\right\|_\infty\nu(Y)\right)^{1/2}\\
& \leq & \|u\|_\infty\|v\|_\infty r(Q)^{1/2}  N(D)^{1/2} \left(\mu(X) + \nu(Y)\right)^{1/2}.
\end{eqnarray*}
It follows that
\begin{equation}\label{eq_estq}
\|\pi_Q(T)\|\leq \|u\|_\infty\|v\|_\infty r(Q)^{1/2} \left(\mu(X) + \nu(Y)\right)^{1/2} |Q|^{1/2}.
\end{equation}
Now let $\kappa$ be a quasi-diagonal set and
$(Q_n)_{n\in \bb{N}}$ be a decreasing sequence of elementary subsets of $X\times Y$ such that
$\sup_{n\in \bb{N}} |Q_n| < \infty$, $r(Q_n) \to_{n\to\infty} 0$ and $\kappa \subseteq_{\omega} \cap_{n=1}^{\infty} Q_n$.
In order to show that $\kappa$ is thin, it suffices to see that
\begin{equation}\label{eq_tozero}
\pi_{Q_n}(T)\to_{n\to \infty} 0, \ \ \ T\in \cl K.
\end{equation}
By (\ref{eq_inmd}),
$\sup_{n\in \bb{N}} \|\pi_{Q_n}\| < \infty$.
It thus suffices to assume
that the measures $\mu$ and $\nu$ are finite and
$T = vu^*$, where $u$ (resp. $v$) is a bounded measurable
function on $X$ (resp. $Y$).
The convergence (\ref{eq_tozero}) is now immediate from (\ref{eq_estq}).
\end{proof}

Recall \cite{st1,tod} that a subset $\kappa\subseteq X\times Y$ is called a set
\emph{of finite width} if there exist measurable functions
$f_j : X\to\mathbb R$ and $g_j : Y\to\mathbb R$, $j=1,\ldots,n$, such that
\begin{equation}\label{width}
\kappa = \{(x,y)\in X\times Y: f_j(x)\leq g_j(y), j=1,\ldots, n\};
\end{equation}
the smallest $n$ as in (\ref{width}) is called the \emph{width} of $\kappa$.
It was proved in \cite{st1} and \cite{tod}
that every set of finite width is a set of operator synthesis and hence of compact operator synthesis.
We will see shortly that the class of sets of compact operator
synthesis is stable with respect to forming the union and the
intersection with sets of finite width.

\begin{proposition}\label{thinwidth}
If $\kappa\subseteq X\times Y$ is a set of finite width then $\kappa$ has thin boundary.
\end{proposition}

\begin{proof}
By Remark \ref{r_four} (iv), it suffices to prove that,
if $f : X\to \bb{R}$ and $g : Y\to \bb{R}$ are measurable functions then the set
$\kappa = \{(x,y): f(x) \le g(y)\}$ has thin boundary.
The set
$\Omega \stackrel{def}{=} \{(x,y): f(x) < g(y)\}$ is $\omega$-open, because
$$\Omega = \cup_{r\in \mathbb{Q}} \{x: f(x) < r\} \times\{x: g(x) > r\}.$$
It thus suffices to prove that the set
$\kappa_1 \stackrel{def}{=} \{(x,y): f(x) = g(y)\}$
has thin boundary.
Let
$$V = \left\{a\in \mathbb{R}: \mu\left(\{x : f(x) = a\}\right) > 0\right\}.$$
Since $\mu$ is $\sigma$-finite, $V$ is countable.
Let $X_0 = \cup_{a\in V}\{x : f(x) = a\}$. Then $\kappa_1 \cap (X_0\times Y)$ is $\omega$-open whence
$\kappa_1 \cap \left((X\setminus X_0)\times Y\right)$ is thin by \cite[Lemma 6.2]{gralmul}
and Theorem \ref{newquasi}.
\end{proof}

It was shown in \cite[Corollary 4.2]{eletod} that
the union of a set of finite width and a set of operator synthesis satisfies operator synthesis.
By \cite[Theorem~4.9]{st1}, such a statement
does not hold for intersections instead of unions.
The next corollary, immediate  from Proposition \ref{thinwidth} and
Theorem \ref{intersection}, answers the analogous questions for compact operator synthesis.

\begin{corollary}\label{c_fw}
The intersection (resp. the union) of a set of finite width and a set of compact operator synthesis
is a set of compact operator synthesis.
\end{corollary}

We next note that there exist $U$-sets that are not strongly quasi-diagonal. For a locally compact group $G$ and a closed subset $E\subseteq G$ let
$$E^*=\{(s,t):ts^{-1}\in E\}.$$

\begin{proposition}\label{p_unqd}
Let $E = \left\{\frac{1}{n}: n\in\mathbb N\right\} \cup\{0\}$. Then
the subset $E^*$ of $\bb{R}\times\bb{R}$ is a set of operator uniqueness with respect to the Lebesgue measure
which is not strongly quasi-diagonal.
\end{proposition}

\begin{proof}
By Proposition \cite[Corollary 5.3]{gralmul}, $E$ is a set of uniqueness for $\bb{R}$.
By \cite[Theorem 4.9]{gralmul}, $E^*$ is a set of operator uniqueness for $\bb{R}\times\bb{R}$.
It remains to show that $E^*$ is not strongly quasi-diagonal.

Fix $a\in\mathbb R$ and consider the unitary shift operator $T_a$ on $L^2(\mathbb R)$ given by
$$(T_af)(x)=f(x-a), \ \ \ f\in L^2(\mathbb R).$$
Let $\cl A$ be the algebra generated by $T_a$ and $T_a^*$ and let $\cl A_+$ be its unital subalgebra generated by $T_a$.
We define a projection  $\pi:\cl A\to\cl A_+$  by
$$\pi\left(\sum_{k=-n}^{n}\alpha_kT_a^k\right) = \sum_{k=0}^n\alpha_k T_a^k.$$

\smallskip

\noindent {\it Claim 1.} The transformation $\pi$ is unbounded, if $\cl A$ and $\cl A_+$ are equipped with the operator norm.

\medskip

\noindent {\it Proof of Claim 1.}
Let $S$ be  the shift operator  on $\ell^2(\mathbb Z)$ given by
$Se_k=e_{k+1}$, where $\{e_k\}_{k\in\mathbb Z}$ is the standard basis in $\ell^2(\mathbb Z)$,
and set
$\tilde{\cl A} = {\rm span} \{S^k :  k\in \bb{Z}\}$
and
$\tilde{\cl A}_+ = {\rm span} \{S^k :  k \geq 0\}$.
It follows from the Gelfand-Naimark Theorem that the mapping
$T_a\mapsto S$ is an isometric isomorphism between  $\cl A$ (resp. $\cl A_+$)
and $\tilde{\cl A}$ (resp. $\tilde{\cl A}_+$).
Clearly, $\tilde{\cl A}$ and $\tilde{\cl A_+}$ are subalgebras of $\vn(\mathbb Z)$
(indeed, $S^n=\lambda_n$, $n\in\mathbb Z$).
The projection $\pi$ then corresponds to the map
$$\cl M_\varphi: T\mapsto \varphi\cdot T, \ \ \ T\in\tilde{\cl A},$$
where $\varphi(k)=\chi_{[0,\infty)}(k)$ and $\varphi\cdot S^k=\varphi(k)S^k$, $k\in\mathbb Z$.
It is known that $\varphi\notin M(A(\mathbb Z)) = B(\mathbb Z)$,
as otherwise it would be the Fourier-Stieltjes transform of a measure $\mu\in M(\mathbb T)$,
which is impossible (see e.g. \cite[Corollary 8.3]{Bennett}).
Since $\tilde{\cl A}$ is dense in $\vn(\mathbb Z)$, we obtain that ${\cl M}_\varphi$, and hence $\pi$, is unbounded.

\medskip

For $N\in\mathbb N$,
let $\cl A_N=\text{span}\{T_a^n:|n|\leq N\}$. Then $\cl A_N\subseteq \cl A_{N+1}$,
$N\in \bb{N}$, and $\cl A=\cup_{N\in\mathbb N}\cl A_N$.
By Claim 1, if  $\pi_{N,a}:=\pi|_{\cl A_N}$ then $\|\pi_{N,a}\|\to\infty$ as $N\to\infty$.

For $Z > 0$, let $P_Z = M_{\chi_{[-Z,Z]}}$ be the operator on $L^2(\bb{R})$ of multiplication
by $\chi_{[-Z,Z]}$,
and set
$$S_{n,Z,a}:= P_ZT_a^nP_Z\text{ and } \cl A_{N,Z,a}:=\text{span}\{S_{n,Z,a}:|n|\leq N\}=P_Z{\cl A_N}P_Z.$$
It is easy to see that the map $\pi_{N,Z,a} : \cl A_{N,Z,a} \to \cl A_{N,Z,a}$, given by
$$\pi_{N,Z,a}\left(\sum_{n=-N}^N \alpha_n S_{n,Z,a}\right) = \sum_{n=0}^N \alpha_n S_{n,Z,a},$$
is well-defined.

\medskip

\noindent {\it Claim 2. } The sequence $\left(\|\pi_{L,L,a}\|\right)_{L\in\mathbb N}$ is unbounded.

\medskip

\noindent {\it Proof of Claim 2.}
As $P_Z\to_{Z\to+\infty} I$ strongly,
$\pi_{N,Z,a}(P_ZXP_Z)\to_{Z\to+\infty} \pi_{N,a}(X)$ weakly for every $X\in \cl A_N$.
Thus,
$\|\pi_{N,a}\|\leq \sup_{Z > 0} \|\pi_{N,Z,a}\|$.
As $\|\pi_{N,a}\|\to\infty$, for any  positive $K$
there exist $N\in \bb{N}$ and $Z(N) > 0$ such that $\|\pi_{N,Z(N),a}\| > K$.
Clearly, $\|\pi_{N,Z,a}\|\geq\|\pi_{M,W,a}\|$ whenever $N\geq M$ and $Z\geq W$. Therefore
$\|\pi_{L,L,a}\|>K$ for any $L\in\mathbb N$, $L\geq \max\{N, Z(N)\}$,
establishing the claim.

\medskip

\noindent  {\it Claim 3.}
$\|\pi_{N,Z,a}\|=\|\pi_{N, Zk, ak}\|$, for all $a,Z,k>0$ and all $N\in \mathbb N$.

\medskip

\noindent  {\it Proof of Claim 3.}
Let $U : L^2(\mathbb R)\to L^2(\mathbb R)$ be the unitary operator
given by
$\displaystyle (Uf)(x)= \frac{1}{\sqrt{k}}f\left(\frac{x}{k}\right)$, $f\in L^2(\mathbb R)$. We have
\begin{eqnarray*}
(UP_ZT_aP_ZU^{-1}f)(x)&=&\chi_{[-kZ,kZ]}\chi_{[-kZ+ka,kZ+ka]}f(x-ka)\\
&=&(P_{Zk}T_{ak}P_{Zk}f)(x).
\end{eqnarray*}
Hence $U\cl A_{N,Z,a}U^{-1}=\cl A_{N,Zk,ak}$ and $U\pi_{N,Z,a}(X)U^{-1}=\pi_{N,Zk,ak}(UXU^{-1})$
for any $X\in \cl A_{N,Z,a}$, implying the statement.

\medskip

Assume now that $E^*$ is strongly quasi-diagonal and let $(Q_n)_{n\in \bb{N}}$ be a decreasing sequence of elementary subsets of
$\bb{R}\times\bb{R}$ such that $E^* \simeq \cap_{n=1}^{\infty} Q_n$ and
$C := \sup_{n\in \bb{N}} |Q_n| < \infty$.
Set $F_a=\{(x,y): y - x = a\}$, $a\in \bb{R}$.
Note that the operator $T_a$, and hence $P_ZT_aP_Z$, $Z>0$, is supported
by $F_a$, $a\in \bb{R}$.
Hence $\pi_{Q_n}(T_a)=T_a$ if $a=1/m$, $m\in\mathbb N$, or $a = 0$.

Let $a < 0$. Given $Z>0$, we have $F_a\cap([-Z,Z]\times[-Z,Z])\subseteq \cup_{n\in\bb{N}} Q_n^c$.
By Lemma \ref{phi-zero},
given $\varepsilon>0$, there exist $X_\varepsilon, Y_\varepsilon\subseteq [-Z,Z]$
and a finite subset $N(\varepsilon)\subseteq \bb N$  such that
$m([-Z,Z]\setminus X_\varepsilon) < \varepsilon$, $m([-Z,Z]\setminus Y_\varepsilon) < \varepsilon$ and
$F_a\cap(X_\varepsilon\times Y_\varepsilon)\subseteq \cup_{n\in N(\varepsilon)} Q_n^c$.
Thus, there exists
$n(Z,\varepsilon) \in \bb{N}$ such that $M_{\chi_{Y_\varepsilon}}\pi_{Q_n}(T_a)M_{\chi_{X_\varepsilon}}=0$
 for all $n\geq n(Z,\varepsilon)$.

Let $Z = \frac{1}{(k-1)!}$, $N=k$ and $a=\frac{1}{k!}$.
We have that $S_{l,Z,a}$, $|l|\leq N$, is supported by $$\left\{\frac{m}{k!}: -k \leq m \leq k\right\}^*\cup\{0\}^*$$
which is a subset of $\left\{\frac{1}{r}: r\in\mathbb Z, r\neq 0\right\}^*\cup \{0\}^*$.   Hence, given $\varepsilon >0$,
by the previous paragraph, there exist $n=n(k,\varepsilon)$ and subsets $X_\varepsilon$, $Y_\varepsilon$ of $[-Z,Z]$ such that $m([-Z,Z]\setminus X_\varepsilon)$, $m([-Z,Z]\setminus Y_\varepsilon)<\varepsilon$ and
$$M_{\chi_{Y_\varepsilon}}\pi_{Q_{n(k,\varepsilon)}}\left(\sum_{l=-k}^k\alpha_l S_{l,Z,a}\right)M_{\chi_{X_\varepsilon}}
= M_{\chi_{Y_\varepsilon}}\sum_{l=0}^k\alpha_l S_{l,Z,a}M_{\chi_{X_\varepsilon}}.$$
Therefore,
\begin{equation}\label{eq_LR}
\|\pi_{X_{\varepsilon}\times Y_{\varepsilon}}
\pi_{Q_{n(k,\varepsilon)}}\|
\geq \|\pi_{X_{\varepsilon}\times Y_{\varepsilon}}\pi_{k,1/(k-1)!,1/k!}\|.
\end{equation}

We claim now that the sequence $(\|\pi_{Q_n}\|)_{n\in \mathbb N}$ is unbounded. In fact, assuming the contrary,
(\ref{eq_LR}) implies that the set
$$\left\{\|\pi_{X_{\varepsilon}\times Y_{\varepsilon}}\pi_{k,1/(k-1)!,1/k!}\| : k\in \bb N, \varepsilon > 0\right\}$$
is bounded.
Since
$$M_{\chi_{Y_\varepsilon}}\pi_{k,1/(k-1)!,1/k!}(T)M_{\chi_{X_\varepsilon}} \to_{\varepsilon\to 0} \pi_{k,1/(k-1)!,1/k!}(T)$$
weakly for every $T$,
the set $(\|\pi_{k,1/(k-1)!,1/k!}\|)_{k\in\mathbb N}$ is also bounded.
By Claim 3, $\|\pi_{k,1/(k-1)!,1/k!}\|=\|\pi_{k,k,1}\|$, contradicting Claim 2.
Hence $(\|\pi_{Q_n}\|)_{n\in \mathbb N}$ is unbounded. As $|Q_n|\geq\|\pi_{Q_n}\|$, this contradicts the assumption that $\sup_{n\in\mathbb N}|Q_n|<\infty$. Therefore $E^*$ is not strongly quasi-diagonal.
\end{proof}

Let $(Z,\sigma)$ be a standard measure space,
$f : X\to Z$, $g : Y\to Z$ be measurable maps, and suppose that $f$
is not constant on any set of positive measure.
Assume that the measures $\mu$ and $\nu$ are finite.
It was shown in \cite[Lemma 6.2]{gralmul}
that, in this case, the set $\{(x,y) : f(x) = g(y)\}$ is quasi-diagonal.
It turns out that this class of quasi-diagonal sets is not exhausting.
Indeed, any set of the form $\{(x,y) : f(x) = g(y)\}$ is a set of operator synthesis
(see \cite{st1} and \cite{tod}), while we will
next show that this does not extend to quasi-diagonal sets.

\begin{proposition}\label{p_qdnos}
There exists a quasi-diagonal set which is not a set of operator synthesis.
\end{proposition}

\begin{proof}
Let
$$\kappa = \{((x_1,y),(x_2,y)) : x_i\in\mathbb T, i = 1,2, \ y\in \mathbb T\}.$$
If $f : \mathbb T^2\to\mathbb T$
is the function given by
$f(x,y) = y$, then
$$\kappa = \{((x_1,y_1),(x_2,y_2)) \in \mathbb T^2\times\mathbb T^2 : f(x_1,y_1) = f(x_2,y_2)\}$$
and, by \cite[Lemma 6.2]{gralmul}, $\kappa$ is quasi-diagonal.
Let $K\subseteq \mathbb T$ be a subset that fails spectral synthesis in $\mathbb T$.
Using the fact that local spectral synthesis passes to supergroups (see Remark \ref{rem} (ii)),
we have that $K\times\{0\}\subseteq \mathbb T\times\{0\}$ is not a set of spectral synthesis for $\mathbb T^2$.

By \cite{froelich}, $(K\times\{0\})^*$ is not a set of operator synthesis. On the other hand,
$(K\times\{0\})^*\subseteq (\mathbb T\times\{0\})^* = \kappa$, and is hence a quasi-diagonal set.
\end{proof}


\section{Connection between compact operator synthesis and reduced spectral synthesis}\label{s_c}

Let $G$ be a second countable locally compact group, equipped with left Haar measure.
For a function $u : G\to \bb{C}$, let
$N(u) : G\times G \to \bb{C}$ be the function given by
$$N(u)(s,t) = u(ts^{-1}).$$
By \cite{bf, j}, if $u$ belongs to  the algebra $M^{\rm cb}A(G)$ of
completely bounded multipliers of $A(G)$
then $N(u)$ is a Schur multiplier.
In \cite{gralmul} we introduced a symbolic calculus for normal completely bounded maps acting on $\cl B(L^2(G))$.
Namely, for every $\varphi\in {\cl T}(G)$ and $T\in \cl B(L^2(G))$,
we let $E_\varphi(T)\in \vn(G)$ be the operator defined by
the identities
$$\langle E_\varphi(T),u\rangle=\langle T,\varphi N(u)\rangle, \ \ \  u\in A(G),$$
where the first pairing is between $\vn(G)$ and $A(G)$, while the second one is between $ \cl B(L^2(G))$ and ${\cl T}(G)$.
The transformation $\varphi\mapsto E_\varphi$ is contractive as a map from ${\cl T}(G)$
into the space ${\rm CB}^{w^*}(\cl B(L^2(G)),\vn(G))$ of all weak* continuous completely bounded linear maps
from $\cl B(L^2(G))$ into $\vn(G)$, and
$E_\varphi(\cl K(L^2(G)))\subseteq C_r^*(G)$ for any $\nph\in {\cl T}(G)$ \cite[Theorem 4.6]{gralmul}.
We will use this result to clarify the connection between reduced spectral synthesis and compact operator synthesis.
Recall that, for $E\subseteq G$, we set
$$E^* = \{(s,t)\in G\times G : ts^{-1}\in E\}.$$

\begin{theorem}\label{ess-comp}
Let $G$ be a second countable locally compact group and $E\subseteq G$ be a closed subset.
The following are equivalent:

(i) \ $E$ is a set of reduced local spectral synthesis;

(ii) $E^*\subseteq G\times G$ is a set of compact operator synthesis.
\end{theorem}

\begin{proof}
(i)$\Rightarrow$(ii)
Assume that $E\subseteq G$ is a set of reduced local spectral synthesis and
let $T\in \mathfrak M_{\max}(E^*)\cap\cl K$.

\smallskip

\noindent {\it Claim 1. } If $u\in I(E)\cap C_c(G)$ then $S_{N(u)}(T) = 0$.

\smallskip

\noindent {\it  Proof of Claim. }
Fix $u\in I(E)\cap C_c(G)$ and let $\nph\in {\cl T}(G)$.
By \cite[Theorem 4.6]{gralmul}, $E_\varphi(T)\in C_r^*(G)$.
By the proof of \cite[Theorem 4.9]{gralmul}, $E_\varphi(T) \in J(E)^{\perp}$, whenever $\nph = a\otimes b$
for some $a,b\in L^2(G)$. Since the mapping
from ${\cl T}(G)$ into $\vn(G)$ sending
$\psi$ to $E_\psi(T)$, is linear and continuous
\cite[Theorem 4.6]{gralmul}, we have that $E_\varphi(T) \in J(E)^{\perp}$.
Therefore
$$0=\langle E_\varphi(T),u\rangle=\langle T, N(u) \varphi\rangle= \langle S_{N(u)}(T),\varphi\rangle;$$
since this holds for an arbitrary $\nph\in {\cl T}(G)$, we have that $S_{N(u)}(T) = 0$.

\smallskip

\noindent {\it Claim 2. }
If $u\in M^{\rm cb}A(G)$ vanishes on $E$ then $S_{N(u)}(T) = 0$.

\smallskip

\noindent {\it  Proof of Claim. }
Let $u\in M^{\rm cb}A(G)$ vanish on $E$
and $v\in A(G)\cap C_c(G)$. Then $uv\in I(E)\cap C_c(G)$ and, by Claim 1, for every
$F\in {\cl T}(G)$ we have
$$\langle S_{N(u)}(T), N(v)F\rangle = \langle S_{N(uv)}(T),F\rangle = 0.$$
Consider the set $\cl S = \{N(w)F : w\in A(G)\cap C_c(G), F\in {\cl T}(G)\}$.
It is straightforward to verify that
$\cl S$ is invariant under Schur multipliers. Moreover,
the null set of $\cl S$, as defined on p. 296 of \cite{st1}, is marginally equivalent to
the empty set. By \cite[Corollary 4.3]{st1}, ${\rm span}\cl S$ is dense in ${\cl T}(G)$.
It follows that $S_{N(u)}(T) = 0$.

\smallskip

\noindent {\it Claim 3. }
If $w\in \frak{S}(G,G)$ is a Schur multiplier that vanishes on $E^*$ and $K\subseteq G$ is a compact set
then $S_{w}(P(K)TP(K)) = 0$.

\smallskip

\noindent {\it  Proof of Claim. }
Set $T_K = P(K)TP(K)$.
Fix a Schur multiplier $w$ that vanishes on $E^*$, let $\nph\in {\cl T}(G)$ and
set $h = w\nph$. Clearly, $h$ belongs to ${\cl T}(G)$ and vanishes on $E^*$,
and it suffices to show that $\langle T_K, h\rangle = 0$.
By \cite[Lemma 3.13]{akt}, we may assume that $h$ is in addition a Schur multiplier
and $h = h\chi_{L\times M}$ for some compact subsets $L$ and $M$ of $G$.

For each irreducible representation $\pi$ of $G$ of dimension $d_{\pi}$,
the functions $h_{i,j}^{\pi}$ and $\tilde{h}_{i,j}^{\pi}$, defined by \cite[(4.4)]{lt},
vanish on $E^*$, for all $i,j \in \{1,\dots,d_{\pi}\}$ (or $i,j\in \bb{N}$ if $d_{\pi} = \infty$).
Moreover, by \cite[Lemma 3.10]{akt}, $\tilde{h}_{i,j}^{\pi}\in \frak{S}(G,G)$; thus,
as, for any $r\in G$, we have that $\tilde{h}_{i,j}^\pi(sr,tr)=\tilde{h}_{i,j}^\pi(s,t)$ for almost all $(s,t)$,
$\tilde{h}_{i,j}^{\pi} = N(u_{i,j}^{\pi})$ for some element $u_{i,j}^{\pi}\in
M^{\rm cb}A(G)$, vanishing on $E$ \cite{att}.
By Claim~2,
$$S_{\tilde{h}_{i,j}^{\pi}}(T_K) = P(K) S_{\tilde{h}_{i,j}^{\pi}}(T) P(K) = 0.$$
Let $f,g\in L^{\infty}(G)$ be compactly supported. Then $f\otimes g\in \frak{S}(G,G)$ and
$$\langle S_{f\otimes g}(T_K),\tilde{h}_{i,j}^{\pi}\chi_{K\times K} \rangle
=  \langle S_{\tilde{h}_{i,j}^{\pi}}(T_K) f, \overline{g}\rangle = 0.$$
By
\cite[Lemma 3.8]{akt}, $h_{i,j}^{\pi}\chi_{K\times K} \in \frak{S}(G,G)$ and,
by \cite[Lemma 3.12]{akt},
$$\langle S_{f\otimes g}(T_K),h_{i,j}^{\pi}\chi_{K\times K} \rangle = 0.$$
By \cite[Lemma 3.14]{akt},
$$
\langle T_K,h\rangle = \langle S_{\chi_K\otimes \chi_K}(T_K), h\rangle = 0.$$

\smallskip

By Claim 3 and \cite[Proposition~5.3]{st2}, $T_K \in \mathfrak M_{\min}(E^*)$.
Since $G$ is $\sigma$-compact, we conclude that $T\in \mathfrak M_{\min}(E^*)$.

(ii)$\Rightarrow$(i)
Suppose that $E^*$ is a set of compact operator synthesis and let $S\in C_r^*(G)$ with $\supp_{\vn}(S)\subseteq E$.
By \cite[Lemma 4.8]{gralmul}, $S$ is supported by $E^*$.
It is straightforward to see that, if $K\subseteq G$ is a compact set and
$f\in L^1(G)$ is compactly supported, then $P(K)\lambda(f)P(K)$ is a
Hilbert-Schmidt operator. It follows that the operator $S_K := P(K)S P(K)$ is compact.
Since $E^*$ satisfies compact operator synthesis, $S_K \in \frak{M}_{\min}(E^*)$.
Let $u\in I(E)$ be compactly supported. If $\nph\in \cl T(G)$
then $N(u)\nph \in \cl T(G)$ and vanishes on $E^*$.
It follows from (\ref{eq_minmax}) that
$\langle S_K, N(u)\varphi\rangle = 0$ and therefore
$$ P(K)(u\cdot S)P(K) = S_{N(u)}(S_K) = 0.$$
Since $K$ is an arbitrary compact set, $u\cdot S = 0$.
Let $v\in A(G)$ be such that $v = 1$ on the support of $u$. Then
$$\langle S,u\rangle = \langle S,uv\rangle = \langle u\cdot S,v\rangle = 0.$$
Thus, $E$ is a set of reduced local spectral synthesis.
\end{proof}

\begin{remark}\rm
We note that if $A(G)$ has a (not necessarily bounded) approximate identity,
the word "local" can be removed from the statement of the previous theorem.
\end{remark}

\begin{remark}\rm
It is known that the (closed) unit ball and the complement of the open unit ball in ${\mathbb R}^n$ are sets of spectral, and hence of
reduced spectral, synthesis (see e.g. \cite{herz}), but, as shown in Proposition \ref{noness},
if $n\geq 4$, their intersection, namely,
the sphere $S^{n-1}$, is not a set of reduced spectral synthesis.
It hence follows from Theorem~\ref{ess-comp} that the intersection of two sets of compact operator
synthesis is not necessarily a set of compact operator synthesis.
\end{remark}

Let $G$ be a second countable locally compact group and $\omega$
be a continuous homomorphism from $G$ into the multiplicative group $\bb{R}_+$ of positive reals.
For $r\in \bb{R}_+$, let
$$E_{\omega}^r = \{s\in G : \omega(s) \leq r\}.$$
Call a subset $E\subseteq G$ \emph{finitely presented}
if there exist continuous homomorphisms $\omega_i : G\to \bb{R}_+$
and values $r_i\in \bb{R}_+$, $i = 1,\dots,n$, such that
$$E = E_{\omega_1}^{r_1} \cap \cdots \cap E_{\omega_n}^{r_n}.$$

\begin{corollary}\label{c_fwg}
Let $G$ be a second countable locally compact group, $E\subseteq G$ be a finitely presented set
and $F\subseteq G$ be a set of reduced local spectral synthesis. Then $E\cup F$ and $E\cap F$
are sets of reduced local spectral synthesis.
\end{corollary}

\begin{proof}
By Theorem \ref{ess-comp} (resp. Corollary \ref{c_fw}),
$F^*$ (resp. $E^*$) satisfies compact operator synthesis.
By Corollary \ref{c_fw} again,
$(E\cap F)^* = E^* \cap F^*$ and
$(E\cup F)^* = E^*\cup F^*$ are sets of compact operator synthesis in $G\times G$.
By Theorem \ref{ess-comp}, $E\cap F$ and $E\cup F$ are sets of reduced local spectral synthesis.
\end{proof}

\begin{corollary}\label{c_cosvsos}
There exists a set of compact operator synthesis that does not satisfy operator synthesis.
\end{corollary}

\begin{proof}
By Example \ref{ex_ess}, there exists a closed subset $E\subseteq \bb{T}$ of reduced spectral synthesis
that does not satisfy spectral synthesis. By Theorem \ref{ess-comp}, $E^*$ is a set of compact operator synthesis,
and by \cite[Corollary 4.4]{lt}, $E^*$ does not satisfy operator synthesis.
\end{proof}

In view of Corollary \ref{c_cosvsos}, it is natural to ask which measure spaces admit sets
that fail operator synthesis. We provide an answer in the next proposition.

\begin{proposition}\label{p_charfail}
Let $(X,\mu)$ and $(Y,\nu)$ be standard measure spaces. The following are equivalent:

(i) \ there exists an $\omega$-closed subset of $X\times Y$ that fails compact operator synthesis;

(ii) the measures $\mu$ and $\nu$ are not atomic.
\end{proposition}
\begin{proof}
(ii)$\Rightarrow$(i) Let $X_0\subseteq X$ (resp. $Y_0 \subseteq Y$) be a
(measurable) subset with the property that
$\mu|_{X_0}$ (resp. $\nu|_{Y_0}$) is continuous, while
$\mu|_{X_0^c}$ (resp. $\nu|_{Y_0^c}$) is atomic.
There exists a Borel isomorphism $\nph : X_0 \to \bb{T}$
(resp. $\psi : Y_0 \to \bb{T}$) \cite[Theorem 17.41]{kechris},
null-preserving in both directions.
By Theorem \ref{th_every}, there exists a subset $E\subseteq \bb{T}$
that fails reduced spectral synthesis. By Theorem \ref{ess-comp}, the subset $E^*$ of $\bb{T}\times\bb{T}$ fails compact operator synthesis.
By Theorem \ref{pr_invim}, the set
$$\kappa = \{(x,y)\in X_0 \times Y_0 : (\nph(x),\psi(y))\in E^*\}$$
fails compact operator synthesis.
It is now easy to see that $\kappa$, when considered as a subset of $X\times Y$, is not a set of compact operator synthesis.

(i)$\Rightarrow$(ii)
Suppose that $\mu$ is atomic.
Without loss of generality, we may assume that $X = \bb{N}$, equipped with counting measure.
We write $\{e_k\}_{k\in \bb{N}}$
for the canonical orthonormal basis of $\ell^2 = H_1$.
Let $\kappa \subseteq X\times Y$ be
$\omega$-closed and write $\kappa = \cup_{k\in \bb{N}} \{k\}\times Y_k$, where $Y_k\subseteq Y$ are measurable.
Let $P_k$ be the projection with range $\bb{C} e_k$ and $Q_k = P(Y_k)$, $k\in \bb{N}$.
Suppose that $T\in \frak{M}_{\max}(\kappa)$. Then $T = $ w$^*$-$\lim_{n\to\infty} \sum_{k=1}^n Q_k T P_k$, and
$Q_kTP_k\in \frak{M}_{\max}(\kappa)$. Since $Q_kTP_k$ has rank one, we have by \cite{arv} that $Q_kTP_k\in \frak{M}_{\min}(\kappa)$. It follows that $T\in \frak{M}_{\min}(\kappa)$. Thus, $\kappa$ is a set of operator synthesis, and hence a set of compact operator synthesis.
\end{proof}

As noted in Remark (ii) before Proposition \ref{p_bo}, every set of operator uniqueness
satisfies compact operator synthesis.
The converse is not true in general; for example, subsets with
non-empty $\omega$-interior are necessarily sets of operator multiplicity,
although they may satisfy compact operator synthesis
(as follows, from example, from \cite{froelich} and \cite[Corollary, p. 498]{arv}).
In the next proposition we show that this may hold even for sets with empty $\omega$-interior.

\begin{proposition}\label{p_cstb}
There exists an $\omega$-closed set with empty $\omega$-interior that is a
set of operator multiplicity and satisfies compact operator synthesis.
\end{proposition}

\begin{proof}
By \cite{varopoulos}, the sphere $S^2\subseteq \mathbb R^3$ is a set of reduced spectral synthesis
that is also a set of multiplicity \cite{varopoulos}.
By \cite[Theorem 4.9]{gralmul} and Theorem \ref{ess-comp},
$$\left(S^2\right)^* \stackrel{def}{=} \left\{(x,y)\in \bb{R}^3\times \bb{R}^3 : y - x \in S^2\right\}$$
is a set of compact operator synthesis which is a set of operator multiplicity.
By \cite[Theorem 5.2]{llt}, the $\omega$-interior of $\left(S^2\right)^*$ coincides with
$\left({\rm int}(S^2)\right)^*$, and is hence empty.
\end{proof}


\section{Applications to operator equations}\label{s_aoe}

In this section, we apply our results to problems about
compact solutions of operator equations.
Our starting point is the well-known Fuglede-Putnam theorem which
states that, if $N_1$ and $N_2$ are normal operators acting on a Hilbert space $H$,
then the equations
$$N_1T-TN_2=0\text{ and } N_1^*T-TN_2^*,$$
with respect to the operator variable $T\in \cl B(H)$, are equivalent.
Let $\{A_i\}_{i=1}^m$ and $\{B_i\}_{i=1}^m$ be commutative families of normal operators in $\cl B(H)$.
A natural question that arises is whether the equations
$\sum_{i=1}^m A_i$ $T B_i = 0$ and  $\sum_{i=1}^m A_i^* T B_i^*=0$ are equivalent.
It is clear that the Fuglede-Putnam theorem gives an affirmative answer to this question, in the special case where
$m = 2$, $A_1 = N_1$, $A_2 = I$, $B_1 = I$ and $B_2 = N_2$.

Letting $\Delta$ and $\tilde\Delta$ denote the elementary operators on $\cl B(H)$ given by
$$\Delta(T) = \sum_{i=1}^m A_iT B_i \ \mbox{ and } \ \tilde\Delta(T) = \sum_{i=1}^m A_i^* T B_i^*,$$
the question formulated in the previous paragraph is equivalent to whether the equality
\begin{equation}\label{kernel}
\ker\Delta = \ker\tilde\Delta
\end{equation}
holds.
Note that $\ker\Delta|_{\cl C_2} = \ker\tilde\Delta|_{\cl C_2}$, where $\Delta|_{\cl C_2}$
(resp. $\tilde{\Delta}|_{\cl C_2}$) is the restriction of $\Delta$ (resp. $\tilde{\Delta}$)
to the class $\cl C_2$ of Hilbert-Schmidt operators on $H$.
Indeed, $\Delta|_{\cl C_2}$ is a normal operator on the Hilbert space $\cl C_2$ and its adjoint coincides with $\tilde\Delta|_{\cl C_2}$.
In particular, identity
(\ref{kernel}) holds in the case  $H$ is finite-dimensional.
However, if the space $H$ is infinite-dimensional,
(\ref{kernel}) fails and an example, based on a modification of the Schwartz example \cite{Schwartz},
was constructed by the first author in \cite{sh_example}. Using Proposition \ref{noness},
we will show that
(\ref{kernel}) is not valid when instead of $\Delta$ and $\tilde{\Delta}$ we consider their restrictions
to the space of compact operators as well as to the Schatten classes $\cl C_p$ for $p > 2$.

We note that here we are interested in elementary operators;
examples of operators {\it of infinite length} $\Delta: T \to\sum_{i=1}^\infty A_i T B_i$ with
$\ker\Delta|_{\cl C_p}\ne\ker\tilde\Delta|_{\cl C_p}$, $p>2$, were constructed in \cite{st2} using deep results from harmonic analysis.

In what follows, we write $\|T\|_p$ for the Schatten $p$-norm  of an element $T\in\cl C_p(H)$, $1\leq p<\infty$.
We start with a couple of lemmas. Recall that $\mathcal F$ stands for the Fourier transform on
$\mathbb R^n$; thus,  for a measure $\mu\in M(\bb{R}^n)$, the function $\mathcal F(\mu) : \bb{R}^n\to \bb{C}$ is given by
$$\mathcal F(\mu)(x) = \int_{\mathbb R^n} e^{-it\cdot x}d\mu(t), \ \ \ x\in \bb{R}^n.$$
We consider $L^1(\mathbb R^n)$ sitting canonically as an ideal in $M(\bb{R}^n)$.
If $G$ is a locally compact second countable group and $\sigma \in M(G)$, we denote by
$\lambda({\sigma})$
the convolution operator of the measure $\sigma$ \cite[Section 2.5]{Folland_HA},
acting on $L^2(G)$ and given by
$\lambda({\sigma})(\xi) = \sigma\ast\xi$, $\xi\in L^2(G)$.
We note the identity
\begin{equation}\label{eq_fourmeas}
\cl F^{-1} M_{\cl F(\sigma)}\cl F = \lambda({\sigma}), \ \ \ \sigma\in M(\bb{R}^n).
\end{equation}
We also let $L_b$ denote the convolution map $L_b(\xi)=b\ast\xi$, $b\in L^2(\mathbb R^n)$, $\xi\in L^2(\mathbb R^n)$;
note that $L_b(\xi) \in A(\bb{R}^n)$ for all $b,\xi\in L^2(\bb{R}^n)$.

\begin{lemma}\label{inter}
Let $2\leq p\leq +\infty$ and $b\in C_c(\mathbb R^n)$.
If $\psi\in L^p(\mathbb R^n)\cap L^\infty(\mathbb R^n)$ then
the operator $\mathcal F^{-1}M_{\psi}\mathcal F M_b$ lies in $\mathcal C_p$.
Moreover, there exists
$C(b,p) > 0$ such that
$$\left\|\mathcal F^{-1}M_{\psi}\mathcal F M_b\right\|_p\leq C(b,p) \left\|\psi\right\|_p, \ \ \ \psi\in L^p(\mathbb R^n)\cap L^\infty(\mathbb R^n).$$
\end{lemma}

\begin{proof}
We use complex interpolation.  Write $\cl B=\cl B(L^2(\mathbb R^n))$ for brevity.
Recall \cite{berg-lofstrom, pietsch-triebel} that
$(L^2(\mathbb R^n), L^\infty(\mathbb R^n))$ and $(\mathcal C_2,\mathcal B)$ are compatible pairs and $L^p(\mathbb R^n)$ and $\mathcal C_p$
coincide with the interpolation spaces corresponding to the same value  $\theta = p$ of the interpolation parameter.
Let
$$T : L^2(\mathbb R^n)+L^\infty(\mathbb R^n)\to \mathcal B$$
be the linear operator defined by letting
$$
T(\psi) =
  \begin{cases}
    \mathcal F^{-1}M_{\psi}\mathcal F M_b, & \mbox{if } \psi\in L^\infty(\mathbb R^n),\\
    L_{\mathcal F^{-1}(\psi)}M_b, & \mbox{if } \psi\in L^2(\mathbb R^n).
  \end{cases}
$$
 We note first that if $\psi\in L^2(\mathbb R^n)$ then $L_{\mathcal F^{-1}(\psi)}M_b$
 is a Hilbert-Schmidt, and hence a bounded, operator on $L^2(\mathbb R^n)$.
 In fact, it is an  integral operator with square integrable integral
kernel $k$ given by
$$k(x,y) = (\cl F^{-1}(\psi))(x-y)b(y)$$
and hence
\begin{equation}\label{eq_hsn}
\|L_{\mathcal F^{-1}(\psi)}M_b\|_2=\|k\|_2\leq \left\|\psi\right\|_2 \left\|b\right\|_2.
\end{equation}

We claim that the mapping $T$ is well-defined.
To see this, we need to show that, for
 $\psi\in L^2(\mathbb R^n)\cap L^\infty(\mathbb R^n)$ and $\xi,\eta\in L^2(\bb{R}^n)$, we have
\begin{equation}\label{equa}
(\cl F^{-1}M_\psi\cl FM_b\xi,\eta)=(L_{\cl F^{-1}(\psi)} M_b\xi,\eta).
\end{equation}
Since both operators are bounded,
it suffices to prove (\ref{equa}) for all $\xi \in L^2(\bb{R}^n)$ and all
$\eta$ in the dense subspace $\cl L =  L^2(\mathbb R^n)\cap L^1(\mathbb R^n)$ of $L^2(\mathbb R^n)$.
Note that
\begin{eqnarray*}
|(\cl F^{-1}M_{\psi}\cl FM_b\xi,\eta)|
& = & |(M_{\psi}\cl F (b\xi),\cl F(\eta))|\\
& \le &
 \int_{\mathbb R^n}|\psi(x)||\cl F(b\xi)(x)||\cl F(\eta)(x)|dx\\
& \le &
\|\psi\|_2\|\cl F(b\xi)\cl F(\eta)\|_2
\le  \|\psi\|_2\|\cl F(b\xi)\|_2\|\cl F(\eta)\|_\infty.
\end{eqnarray*}
Now assuming that $\eta \in \cl L$ we see that both sides of (\ref{equa})
are continuous as functions of $\psi \in L^2(\mathbb R^n)$.

Set $\cl M = \cl F(L^2(\mathbb R^n)\cap L^1(\mathbb R^n) )$ and note that $\cl M$ is dense in
$L^2(\mathbb R^n)\cap L^{\infty}(\mathbb R^n)$ in $\|\cdot\|_2$.
If $\psi \in \cl M$ then it is the Fourier image of a finite measure so
$\cl F^{-1}M_{\psi}\cl F = L_{\cl F^{-1}(\psi)}$ by (\ref{eq_fourmeas}). Thus the equality (\ref{equa}) is proved.

Note that,
if $\psi\in L^\infty(\mathbb R^n)$ then $T(\psi)\in \cl B$ and $\|T(\psi)\|\leq\|\psi\|_\infty\|b\|_\infty$.
The statement now follows from  (\ref{eq_hsn}) by complex interpolation.
\end{proof}

\begin{lemma}\label{l_sonm}
Let $G$ be a second countable locally compact group, $\mu\in M(G)$ and $u\in A(G)$.
Then $S_{N(u)}(\lambda(\mu)) = \lambda(u\mu)$.
\end{lemma}

\begin{proof}
Suppose that $\mu\in M(G)$ and $u,v \in A(G)$. Then
$$\langle S_{N(u)}(\lambda(\sigma)),v\rangle
= \langle\lambda(\sigma), uv\rangle=\int_G uvd\sigma = \langle\lambda(u\sigma),v\rangle.$$
\end{proof}

In the next lemma, we consider the elements of $M(\bb{R}^n)$ as distributions.
Further, for a given pseudomeasure $Q$ on $\bb{R}^n$, we write $\cl F(Q)$ for the Fourier transform of $Q$.

\begin{lemma}\label{spoper}
Let $\mu$ be the normalised surface area measure on the sphere
$S^{n-1}\subseteq\mathbb R^n$, $Q=\frac{\partial\mu}{\partial x_1}$, $b\in C_c(\mathbb R^n)$ and  $q\in C^\infty(\mathbb R^n)$.
Then
\begin{itemize}
\item[(i)] $\cl F^{-1}M_{\cl F(q\mu)}\cl F M_b\in\cl C_p$, whenever $n > p/(p-2)$;
\item[(ii)] $\cl F^{-1}M_{\cl F(Q)}\cl FM_b\in\cl C_p$, whenever $n > 3p/(p-2)$.
\end{itemize}
\end{lemma}

\begin{proof}
(i) If $a\in C_c^\infty(\mathbb R^n)$ then $qa\in C_c^\infty(\mathbb R^n)$.
If, in addition, $a=1$ on $S^{n-1}$, then  $qa\in  A(\mathbb R^n)$,
$$q\mu = qa\mu\in M(\mathbb R^{n}),$$ and hence
$q\mu \in {\rm PM}(\mathbb R^{n})$.
By (\ref{eq_fourmeas}) and Lemma \ref{l_sonm},
\begin{eqnarray*}
\cl F^{-1}M_{\cl F(q\mu)} \cl F M_b
& = &
\cl F^{-1}M_{\cl F(qa\mu)} \cl F M_b
=
\lambda(qa\mu) M_b\\
& = &
S_{N(qa)}(\lambda(\mu))M_b
= S_{N(qa)}(\cl F^{-1}M_{\cl F(\mu)}\cl FM_b).
\end{eqnarray*}
It follows from the definition of Schur multipliers that
every Schur multiplier is a bounded operator on $\cl C_2$ and $\cl C_\infty$, so by duality,
also bounded on $\cl C_1$. Using a complex interpolation argument, one can then easily show that
Schur multipliers leave the ideal $\cl C_p$ invariant.
It thus suffices to show that
$$\cl F^{-1} M_{\cl F(\mu)} \cl F M_b\in \cl C_p \ \mbox{ if } n > p/(p-2).$$
Since $\cl F(\mu)(x) = O\left(\frac{1}{|x|^{(n-1)/2}}\right)$ as $|x|\to \infty$
(see the proof of Proposition~\ref{noness}),
there exist positive reals $R$, $C$ and $D$
such that
\begin{eqnarray*}
\|\cl F(\mu)\|^p_p
& = &
\int_{|x| < R}|\cl F(\mu)(x)|^pdx+\int_{|x|\geq R}|\cl F(\mu)(x)|^pdx\\
& \leq &
D+C\int_{|x|\geq R}\frac{1}{|x|^{p(n-1)/2}}dx = D + C\int_R^\infty\frac{r^{n-1}}{r^{p(n-1)/2}}dr.
\end{eqnarray*}
If $n > p/(p-2)$ the last integral is convergent and hence
$\cl F(\mu)\in L^p(\mathbb R^n)$. By  Lemma \ref{inter}, $\cl F^{-1}M_{\cl F(\mu)}\cl F M_b\in \cl C_p$.

(ii) The statement can be shown using similar arguments and the fact that $Q$ is a pseudofunction with
$\cl F(Q)(x)=x_1\cl F(\mu)(x)=O\left(\frac{1}{|x|^{(n-3)/2}}\right)$ as $|x|\to\infty$.
\end{proof}

Recall \cite{ivan-luda-pmulti} that an $\omega$-closed set $\kappa\subseteq X\times Y$ is called a set
of \emph{operator $p$-multiplicity}
if $\mathfrak M_{\max}(\kappa)\cap\cl C_p\ne\{0\}$.
We say that  $\kappa\subseteq X\times Y$ is a set of \emph{operator $p$-synthesis}
if $\mathfrak M_{\max}(\kappa)\cap\cl C_p = \mathfrak M_{\min}(\kappa)\cap\cl C_p$.
Clearly every set of compact operator synthesis is a set of operator $p$-synthesis, for any $p \geq 1$.
In the next corollary, we use
Lemma \ref{spoper} to obtain a result on operator $p$-multiplicity and operator $p$-synthesis.

\begin{corollary}
Let $p > 2$. The following statements hold:

(i) \ If $n > p/(p-2)$ then $(S^{n-1})^*$ is a set of operator $p$-multiplicity
in $\bb{R}^n\times \bb{R}^n$;

(ii) If $n > 3p/(p-2)$ then $(S^{n-1})^*$ is not a set of operator $p$-synthesis in $\bb{R}^n\times \bb{R}^n$.
\end{corollary}

\begin{proof}
(i) Let $\mu$ be the normalised surface measure on $S^{n-1}$.
By (\ref{eq_fourmeas}),
$$\supp\mbox{}_{{\rm VN}}\left(\mathcal F^{-1}M_{\cl F(\mu)}\cl F\right) = \supp(\mu) = S^{n-1}.$$
By \cite[Lemma 4.8]{gralmul},
$\supp\left(\mathcal F^{-1}M_{\cl F(\mu)}\cl F\right) \subseteq (S^{n-1})^*$.
The statement now follows from Lemma \ref{spoper} (i).

(ii)
Let $Q = \frac{\partial\mu}{\partial x_1}$, viewed as a pseudomeasure as in the proof of Proposition \ref{noness},
$b\in C_c(\mathbb R^n)$ and $T=\mathcal F^{-1}M_{\cl F(Q)}\cl F M_b$.
By Lemma \ref{spoper}, $T\in \cl C_p$ whenever $n>3p/(p-2)$. Since
$$\supp\mbox{}_{{\rm VN}}\left(\mathcal F^{-1}M_{\cl F(Q)}\cl F\right) = \supp\mbox{}_{{\rm PM}}(Q) \subseteq S^{n-1}$$
(see the proof of Proposition \ref{noness}),
\cite[Lemma 4.8]{gralmul} implies now that
$$\supp\left(\mathcal F^{-1}M_{\cl F(Q)}\cl F M_b\right) \subseteq (S^{n-1})^*,$$
that is,
$T\in\mathfrak M_{\max}((S^{n-1})^*)$.
By the proof of Proposition \ref{noness},
there exists $u\in A(\mathbb R^n)$ vanishing on $S^{n-1}$,
such that $\langle Q,u\rangle\ne 0$. As $A(\mathbb R^n)$ has an approximate unit, $u\cdot Q\ne 0$.
On the other hand, if  $\xi$, $\eta\in  C_c^{\infty}(\mathbb R^n)$ then
\begin{eqnarray*}
\left\langle T,N(u)(\xi\otimes\eta)\right\rangle
& = &
\left\langle S_{N(u)}(\cl F^{-1}M_{\cl F(Q)}\cl FM_b),\xi\otimes\eta\right\rangle\\
& = &
\left(S_{N(u)}(\cl  F^{-1}M_{\cl F(Q)}\cl F)b\xi,\bar\eta\right) = \left\langle u\cdot Q,\bar\eta\ast \check{b\xi}\right\rangle,
\end{eqnarray*}
where we use that $S_{N(u)}(T) = u\cdot T$ whenever $T\in \vn(\mathbb R^n)$ and $u\in A(\mathbb R^n)$, and the fact that
$\langle  \cl F^{-1}M_{\cl F(Q)}\cl F,v\rangle = \langle Q,v\rangle$ for any $v\in C_c^{\infty}(\mathbb R^n)$.

Since the functions of the form
$\bar\eta\ast \check{b}\check{\xi}$, where $\xi,\eta\in L^2(\mathbb R^n)$ and $b\in C_c(\mathbb R^n)$, are dense in $A(\mathbb R^n)$,
we have that $\langle T,N(u)(\xi\otimes\eta)\rangle\ne 0$ for some $\xi$, $\eta$ and $b$.
Since $N(u)$ vanishes on $(S^{n-1})^*$, we conclude that $T\not\in\mathfrak M_{\min}((S^{n-1})^*)$.
\end{proof}

Now we can prove the failure of the generalised Fuglede-Putnam theorem for Schatten classes.

\begin{theorem}\label{th_schattenc}
Let $p > 2$ and $m > (24p-12)/(p-2)$.
There exist commutative families  $\{A_i\}_{i=1}^m$, $\{B_i\}_{i=1}^m$ of normal operators and an operator $T\in\cl C_p$ such that
\begin{equation}\label{eq_but}
\sum_{i=1}^m A_i T B_i=0 \ \mbox{ and } \ \sum_{i=1}^m A_i^* T B_i^*\ne 0.
\end{equation}
Moreover, a compact operator $T$ for which (\ref{eq_but}) holds exists if $m\ge 25$.
\end{theorem}
\begin{proof}
For the purposes of the proof, set $\mathcal D(\mathbb R^n) = C_c^\infty(\mathbb R^n)$
and let $\mathcal D'(\mathbb R^n)$ be its dual space, that is, the space of distributions.
As pointed out before Proposition \ref{noness},
$\cl D(\bb{R}^{2n}) \subseteq A({\mathbb R}^{2n})$ and every pseudo-measure on $\bb{R}^n$
can be viewed as an element of $\mathcal D'(\mathbb R^n)$.

Let $q(x_1,\ldots, x_{2n})=\sum_{i=1}^nx_i^2-1+i\left(\sum_{i=n+1}^{2n}x_i^2-1\right)$, $m =6n+1$
and $\alpha_i$, $\beta_i$, $i=1,\ldots,m$, be polynomials such that
$$q(x-y)=\sum_{i=1}^m \alpha_i(x)\beta_i(y), \quad x,y\in{\mathbb R}^{2n}.$$
Let $u$, $v\in \cl D(\bb{R}^{2n})$, $a_i=u\alpha_i$, $b_i=v\beta_i$,
$A_i=M_{a_i}$ and $B_i=M_{b_i}$, $i = 1,\dots,m$.

Let $n\geq 3$, $\mu$ be the normalised surface measure of $S^{n-1}$ and $\nu = \mu\times\mu$.
Let $\hat{x}_i : \bb{R}^n \to \bb{R}$ be the function given by $\hat{x}_i(x) = x_i$, $x\in \bb{R}^n$,
and set
$$L = (1+i) \hat{x}_{n+1}\frac{\partial}{\partial x_1}-(1-i)\hat{x}_1\frac{\partial}{\partial x_{n+1}},$$
considered as a linear transformation on $\cl D'(\mathbb R^{2n})$.
Set $\gamma = {\mathcal F}(\mu)$. Denoting a typical element of the dual of $\mathbb{R}^{2n}$ by $t = (s,r)$,
we have that ${\mathcal F}(\nu)(t) = \gamma(s)\gamma(r)$, and so
$${\mathcal F}(L(\nu))(s,r)
= - (1+i)s_1\gamma(s)\frac{\partial\gamma}{\partial r_1}(r) + (1-i)r_1\gamma(r)\frac{\partial\gamma}{\partial s_1}(s).$$
Assume that $n > 3p/(p-2)$ or, equivalently, that $p(n-3)/2 > n$.

By \cite[p.154]{stein_weiss},
$\gamma(s) = C \pi|s|^{-(n-2)/2}J_{(n-2)/2}(2\pi|s|)$, for some constant $C$.
Using \cite[Theorem 5.1 (pp. 139-140)]{Folland},
we obtain
$$\gamma(s)= O\left(\frac{1}{|s|^{\frac{n-1}{2}}}\right), \ s_1\gamma(s)= O\left(\frac{1}{|s|^{\frac{n-3}{2}}}\right) \mbox{ and }
\frac{\partial\gamma(s)}{\partial s_1}= O\left(\frac{1}{|s|^{\frac{n-1}{2}}} \right).$$
Thus the functions $\hat{s}_1\gamma$ and $\frac{\partial\gamma}{\partial s_1}$ belong to
$L^p(\mathbb{R}^n)$ and so
${\mathcal F}(L(\nu))\in L^p(\mathbb{R}^{2n})$.
By Lemma \ref{inter}, for any $c\in C_c^{\infty}(\bb{R}^{2n})$, we have that
$$T := {\mathcal F}^{-1}M_{\mathcal F (L(\nu))}{\mathcal F}M_{c} \in \cl C_p(L^2(\mathbb R^{2n})).$$

We will now show that one can choose $u$, $v$ and  $c$ in such a way that
(\ref{eq_but}) hold.
Let $\varphi\in \cl D(\bb{R}^{2n})$.
Since $q$ vanishes on $S^{n-1}\times S^{n-1}$, we have that
$$\langle \nu,q L(\varphi)\rangle = \langle \nu,\bar qL(\varphi)\rangle = 0.$$
On the other hand, $L(q) = 0$ and hence $\langle \nu,L(q)\varphi\rangle = 0$.
It follows that
$$\langle L(\nu),q\varphi\rangle=-\langle \nu,L(q\varphi)\rangle = - \langle \nu,q L(\varphi)\rangle - \langle \nu,L(q)\varphi\rangle = 0$$
and
\begin{eqnarray*}
\langle  L(\nu),\bar q\varphi\rangle
& =& -\langle \nu, L(\bar q)\varphi\rangle - \langle \nu,\bar qL(\varphi)\rangle
=  -\langle \nu, L(\bar q)\varphi\rangle\\
& = &
-\langle \nu, 4(1+i)\hat{x}_1 \hat{x}_{n+1} \varphi\rangle
= -4(1+i)\langle \hat{x}_1 \hat{x}_{n+1} \nu,\varphi\rangle.
\end{eqnarray*}
Thus,
\begin{equation}\label{eq_znz}
qL(\nu) = 0 \ \mbox{ and } \  \bar q L(\nu) = -4(1+i) \hat{x}_1 \hat{x}_{n+1} \nu.
\end{equation}
For $\nph,\psi\in \cl D(\bb{R}^{2n})$, we have
 \begin{eqnarray*}
 (T\varphi,\psi)&=&(\mathcal F(L(\nu))\mathcal F (c\varphi),\mathcal F(\psi))
 = \int \mathcal F(L(\nu))(x)\mathcal F(c\varphi)(x)\overline{\mathcal F\psi(x)}dx\\
& = & \int \mathcal F(L(\nu))(x)
 \mathcal F^{-1}(\bar\psi\ast\check{c}\check{\varphi})(x) dx = \langle L(\nu),\bar\psi\ast \check{c}\check{\varphi}\rangle.
 \end{eqnarray*}
Therefore, using (\ref{eq_znz}), for all $\nph,\psi\in \cl D(\bb{R}^{2n})$, we have
\begin{eqnarray*}
\left(\sum_{i=1}^m M_{a_i} T M_{b_i}\varphi,\psi\right)
& = &
\sum_{i=1}^m  \left(T (b_i\varphi),\overline{a_i}\psi\right)\\
& = &
\left(L(\nu), \sum_{i=1}^m  (a_i \bar{\psi}) \ast (\check{b}_i\check{c}\check{\varphi}),\right)\\
& = &
\left\langle L(\nu), \sum_{i=1}^m (\alpha_i u \bar\psi) \ast (\check{\beta}_i \check{v} \check{c} \check{\varphi}) \right\rangle\\
& = &
\left\langle L(\nu),q\left((u\bar\psi) \ast (\check{v} \check{c} \check{\varphi})\right) \right\rangle
= 0.
\end{eqnarray*}
On the other hand, by (\ref{eq_znz}) and the fact that
the span of the set $\{u \psi \ast (\check{v} \check{c} \check{\varphi}): \psi,\varphi, u,v,c\in \cl D(\bb{R}^n)\}$ is dense in $A(\mathbb R^n)$,
there exist $\nph$, $\psi$, $u$, $v$ and $c$ in $\cl D(\bb{R}^n)$ such that
\begin{eqnarray*}
\left(\sum_{i=1}^m M_{a_i}^* T M_{b_i}^*\varphi,\psi\right)
& =&
\left\langle L(\nu),\bar q(\overline{u\psi}\ast \check{\bar v}\check{c}\check{\varphi})\right\rangle\\
& = & -4(1+i)\left\langle \hat{x}_1 \hat{x}_{n+1}\nu, \overline{u \psi} \ast (\check{\bar{v}} \check{c} \check{\varphi}) \right\rangle \ne 0.
\end{eqnarray*}

In our construction $m = 6n+1$ therefore in terms of $m$ we have the condition $[m/6] > 3p/(p-2)$.
For this it suffices that $m/6 > 1+3p/(p-2)$, that is, $m > (24p-12)/(p-2)$.
Now taking $p\to \infty$ we get that the equations  $\sum_{i=1}^m A_i T B_i=0$ and  $\sum_{i=1}^m A_i^* T B_i^*=0$
are non-equivalent on $\cl K$ for $m > 24$.
 \end{proof}

\end{document}